\documentclass[10pt]{article}
\usepackage[numbers]{natbib}
\usepackage{graphicx}
\usepackage{subcaption}
\usepackage{amssymb,latexsym,amsmath,enumerate,verbatim,amsfonts,amsthm}
\usepackage{color}
\usepackage[pdftex,pdfstartview=FitH]{hyperref}
\usepackage{algorithm}
\usepackage{bm}
\usepackage{multirow}
\usepackage{caption}

\hypersetup{colorlinks=false,linkbordercolor={1 1 1},citebordercolor={1 1 1}}

\newtheorem{lemma}{Lemma}[section]
\newtheorem{definition}{Definition}[section]

\newtheorem{theorem}{Theorem}[section]
\newtheorem{corollary}{Corollary}[section]
\newtheorem{proposition}{Proposition}[section]

\def\bR{{\mathbb R}}

\def\dom{{\rm dom}}
\def\nn{{\nonumber}}
\def\dist{{\rm dist}}

\def\Argmin{\mathop{\rm Arg\,min}}

\def\argmin{\mathop{\rm arg\,min}}
\def\liminf{\mathop{\rm lim\,inf}}
\def\rmsd{\hbox{\scriptsize RMSD}}

\title{\sf A refined convergence analysis of pDCA$_e$ with applications to simultaneous sparse recovery and outlier detection}
\date{April 19, 2018}

\author{Tianxiang Liu\thanks{Department of Applied Mathematics, the Hong Kong Polytechnic University, Hong Kong. {\rm Email}: tianxiang.liu@polyu.edu.hk.}
  \and%
  Ting Kei Pong\thanks{Department of Applied Mathematics, the Hong Kong Polytechnic University, Hong Kong.
This author's work was supported in part by Hong Kong Research Grants Council PolyU153005/17p. {\rm Email}: tk.pong@polyu.edu.hk.}
  \and%
  Akiko Takeda\thanks{Department of Creative Informatics,
Graduate School of Information Science and Technology, the University of Tokyo, Tokyo, Japan. This author's work was supported in part by JSPS KAKENHI Grant Number 15K00031. {\rm Email}: takeda@mist.i.u-tokyo.ac.jp.}
\thanks{RIKEN Center for Advanced Intelligence Project, 1-4-1, Nihonbashi, Chuo-ku, Tokyo 103-0027, Japan. {\rm Email}: akiko.takeda@riken.jp.
}
}

\begin{document}
\maketitle

\begin{abstract}
We consider the problem of minimizing a difference-of-convex (DC) function, which can be written as the sum of a smooth convex function with Lipschitz gradient, a proper closed convex function and a continuous possibly nonsmooth concave function. We refine the convergence analysis in \cite{Wen16} for the proximal DC algorithm with extrapolation (pDCA$_e$) and show that the whole sequence generated by the algorithm is convergent when the objective is level-bounded, {\em without} imposing differentiability assumptions in the concave part. Our analysis is based on a new potential function and we assume such a function is a Kurdyka-{\L}ojasiewicz (KL) function. We also establish a relationship between our KL assumption and the one used in \cite{Wen16}. Finally, we demonstrate how the pDCA$_e$ can be applied to a class of simultaneous sparse recovery and outlier detection problems arising from robust compressed sensing in signal processing and least trimmed squares regression in statistics. Specifically, we show that the objectives of these problems can be written as level-bounded DC functions whose concave parts are {\em typically nonsmooth}. Moreover, for a large class of loss functions and regularizers, the KL exponent of the corresponding potential function are shown to be 1/2, which implies that the pDCA$_e$ is locally linearly convergent when applied to these problems. Our numerical experiments show that the pDCA$_e$ usually outperforms the proximal DC algorithm with nonmonotone linesearch \cite[Appendix~A]{Liu17} in both CPU time and solution quality for this particular application.
\end{abstract}

\section{Introduction}

Nonconvex optimization plays an important role in many contemporary applications such as machine learning and signal processing. In the area of machine learning, for example,
nonconvex sparse learning has become a hot research topic in recent years, and a large number of papers are devoted to the study of classification/regression models with nonconvex regularizers for finding sparse solutions; see, for example, \cite{Fan01,Gong13,Zhang10}. On the other hand, in signal processing, specifically in the area of compressed sensing, many nonconvex models have been proposed in recent years for recovering the underlying sparse/approximately sparse signals. We refer the interested readers to \cite{CWB08,C07,CY08,FL09,SCY08,Wang15,Yin15} and references therein for more details.

In this paper, we consider a special class of nonconvex optimization problems: the difference-of-convex optimization problems. This is a class of problems whose objective can be written as the difference of two convex functions; see the monograph \cite{Tuy16} for a comprehensive exposition. Here, we focus on the following model,
\begin{equation}\label{problem}
  \min_{\bm x\in\bR^n}\\ F(\bm x):= f(\bm x) + P_1(\bm x) - P_2(\bm x),
\end{equation}
where $f$ is a smooth convex function with Lipschitz continuous gradient whose Lipschitz continuity modulus is $L > 0$, $P_1$ is a proper closed convex function and $P_2$ is a \emph{continuous} convex function. In typical applications in sparse learning and compressed sensing, the function $f$ is a loss function representing data fidelity, while $P = P_1 - P_2$ is a regularizer for inducing desirable structures (for example, sparsity) in the solution. Commonly used regularizers include:
\begin{enumerate}
  \item [-] $P_1(\bm x) = \lambda\|\bm x\|_1$ and $P_2(\bm x) = \lambda\sum_{i=1}^n\int_0^{|x_i|}\frac{[\min\{\theta\lambda,t\}-\lambda]_+}{(\theta-1)\lambda}dt$, where $\lambda > 0$ and $\theta>2$. This regularizer is known as the smoothly clipped absolute deviation (SCAD) function; see \cite{Fan01,Gong13};
  \item [-] $P_1(\bm x) = \lambda\|\bm x\|_1$ and $P_2(\bm x) = \lambda\sum_{i=1}^n\int_0^{|x_i|}\min\{1,t/(\theta\lambda)\}dt$, where $\lambda > 0$ and $\theta > 1$. This regularizer is known as the minimax concave penalty (MCP) function; see \cite{Gong13,Zhang10};
  \item [-] $P_1(\bm x) = \lambda\|\bm x\|_1$ and $P_2(\bm x) = \lambda\|\bm x\|$, where $\lambda > 0$. This is known as the  $\ell_{1-2}$ regularizer;  see \cite{Yin15};
  \item [-] $P_1(\bm x) = \lambda\|\bm x\|_1$ and $P_2(\bm x) = \lambda\mu\sum_{i=1}^p|x_{[i]}|$, where $x_{[i]}$ denotes the $i$th largest element in magnitude, $\lambda > 0$, $\mu \in (0,1]$ and $p< n$ is a positive integer. We will refer to this as the Truncated $\ell_1$ regularizer; see \cite{Wang15};
  \item [-] $P_1(\bm x) = \lambda\|\bm x\|_1$ and $P_2(\bm x) = \lambda\sum_{i=1}^n[|x_i| - \theta]_+$, where $\lambda > 0$ and $\theta > 0$. This is known as the Capped $\ell_1$ regularizer; see \cite{Gong13}.
\end{enumerate}
Notice that while the $P_2$ in the SCAD and MCP functions are smooth, the $P_2$ in the other three regularizers mentioned above are nonsmooth.

For difference-of-convex optimization problems such as \eqref{problem}, the classical method for solving them is the so-called DC algorithm (DCA), proposed in \cite{PhamLeThi97}. In this algorithm, in each iteration, one majorizes the concave part of the objective by its local linear approximation and solves the resulting convex optimization subproblem. For efficient implementation of this algorithm, one should construct a suitable DC decomposition so that the corresponding subproblems are easy to solve. This idea was incorporated in the so-called proximal DCA \cite{GoTaTo17}, which is a version of DCA that makes use of the following special DC decomposition of the objective in \eqref{problem} (see \cite[Eq.~16]{PhamLeThi98} for an earlier use of such a decomposition in solving trust region subproblems):
\[
F(\bm x) = \left[\frac{L}{2}\|\bm x\|^2 + P_1(\bm x)\right] - \left[\frac{L}{2}\|\bm x\|^2 - f(\bm x) + P_2(\bm x)\right].
\]
The major computational costs of the subproblems in the proximal DCA come from the computations of the gradient of $f$, the proximal mapping of $P_1$ and a subgradient of $P_2$, which are simple for commonly used loss functions and regularizers.
Later, extrapolation techniques were incorporated into the proximal DCA in \cite{Wen16}. The resulting algorithm was called pDCA$_e$, and was shown to have much better numerical performance than the proximal DCA. Convergence analysis of the pDCA$_e$ was also presented in \cite{Wen16}. In particular, when $F$ is level-bounded, it was established in \cite{Wen16} that any cluster point of the sequence generated by pDCA$_e$ is a stationary point of $F$ in \eqref{problem}. Moreover, under an additional smoothness assumption on $P_2$ and by assuming that a certain potential function is a Kurdyka-{\L}ojasiewicz (KL) function, it was further proved that the whole sequence generated by pDCA$_e$ is convergent. Local convergence rate was also discussed based on the KL exponent of the potential function. However, the analysis there heavily relies on the smoothness assumption on $P_2$, which does not hold for many commonly used regularizers such as the Capped $\ell_1$ regularizer \cite{Gong13} and the Truncated $\ell_1$ regularizer \cite{Wang15} mentioned above. More importantly, as we shall see later in Section~\ref{application}, the objectives of models for simultaneous sparse recovery and outlier detection can be written as DC functions whose concave parts are {\em typically nonsmooth}. Thus, for these problems, the analysis in \cite{Wen16} cannot be applied to studying global sequential convergence nor local convergence rate of the sequence generated by pDCA$_e$.

In this paper, we refine the convergence analysis of pDCA$_e$ in \cite{Wen16} to cover the case when the $P_2$ in \eqref{problem} is possibly {\em nonsmooth}. Our analysis is based on a potential function different from the one used in \cite{Wen16}. By assuming this new potential function is a KL function and $F$ is level-bounded, we show that the whole sequence generated by pDCA$_e$ is convergent. We then study a relationship between the KL assumption used in this paper and the one used in \cite{Wen16}. Specifically, under a suitable smoothness assumption on $P_2$, we show that if the potential function used in \cite{Wen16} has a KL exponent of $\frac12$, so does our new potential function. KL exponent is an important quantity that is closely related to the local convergence rate of first-order methods \cite{At09,At10,Li_TK_Calculus}, and we also provide an explicit estimate of the KL exponent of the potential function used in our analysis for some commonly used $F$. Finally, we discuss how the pDCA$_e$ can be applied to a class of simultaneous sparse recovery and outlier detection problems in least trimmed squares regression in statistics (see \cite{Rousseeuw83,RousseeuwLeroy87}) and robust compressed sensing in signal processing (see \cite{CRE2016} and references therein). Specifically, we demonstrate how to explicitly rewrite the objective function as a level-bounded DC function in the form of \eqref{problem}, and show that the KL exponent of the corresponding potential function is $\frac12$ for many simultaneous sparse recovery and outlier detection models: this implies local linear convergence of pDCA$_e$ when applied to these models. In our numerical experiments on this particular application, the pDCA$_e$ always outperforms the proximal DCA with nonmonotone linesearch \cite{Liu17} in both solution quality and CPU time.

The rest of the paper is organized as follows. In Section~\ref{sec2}, we introduce notation and preliminary results. We also review some convergence properties of the pDCA$_e$ from \cite{Wen16}. In Section~\ref{sec3}, we present our refined global convergence analysis for pDCA$_e$. Relationship between the KL assumption used in this paper and the one used in \cite{Wen16} is discussed in Section~\ref{sec4}. In Section~\ref{application}, we describe how our algorithm can be applied to a large class of simultaneous sparse recovery and outlier detection problems. Numerical results are presented in Section~\ref{sec6}.

\section{Notation and Preliminaries}\label{sec2}

In this paper, vectors and matrices are represented in bold lower case letters and
upper case letters, respectively. We use $\bR^n$ to denote the $n$-dimensional Euclidean space with inner product $\langle\cdot, \cdot\rangle$ and the Euclidean norm $\|\cdot\|$. For a vector $\bm x\in\bR^n$, we let $\|\bm x\|_1$ and $\|\bm x\|_0$ denote the $\ell_1$ norm and the number of nonzero entries in $\bm x$ ($\ell_0$ norm), respectively. For two vectors $\bm a$, $\bm b\in\bR^n$, we let ``$\circ$'' denote the Hadamard (entrywise) product, i.e., $(\bm a\circ\bm b)_i = a_ib_i$, $i = 1,...,n$. Given a matrix $\bm A\in\bR^{m\times n}$, we let $\bm A^\top$ denote its transpose, and we use $\lambda_{\max}(\bm A)$ to denote the largest eigenvalue of $\bm A$ when $\bm A$ is symmetric, i.e., $\bm A = \bm A^\top$.

Next, for a nonempty closed set $\mathcal{C}\subseteq\bR^n$, we write $\dist(\bm x, \mathcal{C}):=\inf_{\bm y\in\mathcal{C}}\|\bm x - \bm y\|$ and define the indicator function $\delta_{\mathcal{C}}$ as
\begin{equation*}
  \delta_{\mathcal{C}}(\bm x) =
  \begin{cases}
  0 & {\rm if}\  \bm x\in \mathcal{C},\\
  \infty & {\rm otherwise.}
  \end{cases}
\end{equation*}
For an extended-real-valued function $h: \bR^n\rightarrow[-\infty, \infty]$, its domain is defined as $\dom\, h = \{\bm x\in \bR^n: h(\bm x) < \infty\}$. Such a function is said to be proper if it is never $-\infty$ and its domain is nonempty, and is said to be closed if it is lower semicontinuous. A proper closed function $h$ is said to be level-bounded if $\{\bm x\in\bR^n: h(\bm x) \le \gamma\}$ is bounded for all $\gamma\in\bR$. Following \cite[Definition~8.3]{RW98}, for a proper closed function $h: \bR^n\rightarrow\bR\cup\{\infty\}$, the (limiting) subdifferential of $h$ at $\bm x\in\dom\, h$ is defined as
\begin{equation}\label{limit_subdiff}
\partial h(\bm x)= \left\{\bm v:\; \exists~ \bm x^k \stackrel{h}{\to}\bm x, \bm v^k \to \bm v\  \mbox{with}\ \ \liminf_{\begin{subarray}
  \ \bm y\to \bm x^k\\
   \bm y\neq \bm x^k
\end{subarray}}\frac{h(\bm y) - h(\bm x^k) - \langle\bm v^k, \bm y - \bm x^k\rangle}{\|\bm y - \bm x^k\|}\ge 0\ \forall\ k\right\},
\end{equation}
where $\bm x^k \stackrel{h}{\to}\bm x$ means $\bm x^k \to \bm x$ and $h(\bm x^k)\to h(\bm x)$. We also write $\dom\,\partial h: = \{\bm x\in\bR^n: \partial h(\bm x)\neq\emptyset\}$. It is known that if $h$ is continuously differentiable, the subdifferential \eqref{limit_subdiff} reduces to the gradient of $h$ denoted by $\nabla h$; see, for example, \cite[Exercise~8.8(b)]{RW98}. When $h$ is convex, the above subdifferential reduces to the classical subdifferential in convex analysis, see, for example, \cite[Proposition~8.12]{RW98}. Let $h^*$ denote the convex conjugate of a proper closed convex function $h$, i.e.,
\[
h^*(\bm u) = \sup_{\bm x\in \bR^n}\{\langle \bm u,\bm x\rangle - h(\bm x)\}.
\]
Then $h^*$ is proper closed convex and the Young's inequality holds, relating $h$, $h^*$ and their subgradients: for any $\bm x$ and $\bm y$, it holds that
\[
h(\bm x) + h^*(\bm y) \ge \langle\bm x,\bm y\rangle,
\]
and the equality holds if and only if $\bm y \in \partial h(\bm x)$. Moreover, for any $\bm x$ and $\bm y$, one has $\bm y \in \partial h(\bm x)$ if and only if $\bm x \in \partial h^*(\bm y)$.

We next recall the Kurdyka-{\L}ojasiewicz (KL) property, which is satisfied by many functions such as proper closed semialgebraic functions, and is important for analyzing global sequential convergence and local convergence rate of first-order methods; see, for example, \cite{At09,At10,At13}. For notational simplicity, for any $a\in(0,\infty]$, we let $\Xi_a$ denote the set of all concave continuous functions $\varphi: [0, a)\rightarrow [0,\infty)$ that are continuously differentiable on $(0,a)$ with positive derivatives and satisfy $\varphi(0) = 0$.
\begin{definition}\label{KL_prop}
{\bf (KL property and KL exponent)} A proper closed function $h$ is said to satisfy the KL property at $\bar{\bm x}\in\dom\,\partial h$ if there exist $a\in(0, \infty]$, $\varphi\in\Xi_a$ and a neighborhood $U$ of $\bar{\bm x}$ such that
\begin{equation}\label{kl_defi}
  \varphi'(h(\bm x) - h(\bar{\bm x}))\dist(\bm 0, \partial h(\bm x)) \ge 1
\end{equation}
whenever $\bm x\in U$ and $h(\bar{\bm x}) < h(\bm x) < h(\bar{\bm x}) + a$. If $h$ satisfies the KL property at $\bar{\bm x}\in\dom\,\partial h$ and the $\varphi$ in \eqref{kl_defi} can be chosen as $\varphi(s) = cs^{1-\alpha}$ for some $\alpha\in[0,1)$ and $c > 0$, then we say that $h$ satisfies the KL property at $\bar{\bm x}$ with exponent $\alpha$. We say that $h$ is a KL function if $h$ satisfies the KL property at all points in $\dom\,\partial h$, and say that $h$ is a KL function with exponent $\alpha \in [0,1)$ if $h$ satisfies the KL property with exponent $\alpha$ at all points in $\dom\,\partial h$.
\end{definition}

The following lemma was proved in \cite{Bo14}, which concerns the uniformized KL property. This property is useful for establishing convergence of first-order methods for level-bounded functions.
\begin{lemma}\label{lemma_KL}
{\bf (Uniformized KL property)} Suppose that $h$ is a proper closed function and let $\Gamma$ be a compact set. If $h$ is a constant on $\Gamma$ and satisfies the KL property at each point of $\Gamma$, then there exist $\epsilon, a > 0$ and $\varphi\in\Xi_a$ such that
\begin{equation}\label{kl_defii}
  \varphi'(h(\bm x) - h(\hat{\bm x}))\cdot\dist(\bm 0, \partial h(\bm x)) \ge 1
\end{equation}
for any $\hat{\bm x}\in\Gamma$ and any $\bm x$ satisfying $\dist(\bm x, \Gamma) < \epsilon$ and $h(\hat{\bm x}) < h(\bm x) < h(\hat{\bm x}) + a$.
\end{lemma}

Before ending this section, we review some known results concerning the pDCA$_e$ proposed in \cite{Wen16} for solving \eqref{problem}. The algorithm is presented as
Algorithm~\ref{alg1}.
\begin{algorithm}
\caption{Proximal difference-of-convex algorithm with extrapolation (pDCA$_e$) for \eqref{problem}}\label{alg1}
\begin{description}
\item [Input:] $\bm x^0\in\dom\,P_1$, $\{\beta_k\}\subseteq [0,1)$ with $\sup_k\beta_k < 1$. Set $\bm x^{-1} = \bm x^0$.

\begin{description}
\item [for] $k = 0,1,2,\ldots$

\item [] take any $\bm\xi^{k+1}\in\partial P_2(\bm x^k)$ and set
\begin{equation}\label{iter_update}
  \left\{
   \begin{aligned}
  \bm u^k &= \bm x^k + \beta_k(\bm x^k - \bm x^{k-1}),  \\
  \bm x^{k+1} &= \argmin\limits_{\bm x\in \bR^n}\left\{\langle\nabla f(\bm u^k) - \bm\xi^{k+1}, \bm x\rangle + \frac L2\|\bm x - \bm u^k\|^2 + P_1(\bm x)\right\}.  \\
   \end{aligned}
  \right.
\end{equation}
\item [end for]
\end{description}
\end{description}
\end{algorithm}
This algorithm was shown to be convergent under suitable assumptions in \cite{Wen16}. The convergence analysis there was based on the following potential function; see \cite[Eq.~(4.10)]{Wen16}:
\begin{equation}\label{hatE}
  \hat E(\bm x, \bm w) = F(\bm x) + \frac{L}2\|\bm x- \bm w\|^2.
\end{equation}
This potential function has been used for analyzing convergence of variants of the proximal gradient algorithm with extrapolations; see, for example, \cite{ChaDos15,WCP17}. By showing that the potential function $\hat E$ is nonincreasing along the sequence $\{(\bm x^{k+1},\bm x^k)\}$ generated by the pDCA$_e$, the following subsequential convergence result was established in \cite[Theorem~4.1]{Wen16}; recall that $\bar{\bm x}$ is a stationary point of $F$ in \eqref{problem} if
\begin{equation*}
   \bm 0\in\nabla f(\bar{\bm x}) + \partial P_1(\bar{\bm x}) - \partial P_2(\bar{\bm x}).
\end{equation*}
\begin{theorem}\label{thm11}
Suppose that $F$ in \eqref{problem} is level-bounded. Let $\{\bm x^k\}$ be the sequence generated by {\rm pDCA$_e$} for solving \eqref{problem}. Then the following statements hold.
\begin{enumerate}
  \item [{\rm (i)}] $\lim_{k\rightarrow\infty}\|\bm x^{k+1} - \bm x^k\| = 0$.
  \item [{\rm (ii)}] The sequence $\{\bm x^k\}$ is bounded and any accumulation point of $\{\bm x^k\}$ is a stationary point of $F$.
\end{enumerate}
\end{theorem}

Global sequential convergence of the whole sequence generated by the pDCA$_e$ was established in \cite[Theorem~4.2]{Wen16} by assuming in addition that $\hat E$ is a KL function and that $P_2$ satisfies a certain {\em smoothness} condition.
\begin{theorem}\label{thm12}
Suppose that $F$ in \eqref{problem} is level-bounded and $\hat E$ in \eqref{hatE} is a KL function. Suppose in addition that $P_2$ is continuously differentiable on an open set ${\cal N}$ containing the set of stationary points of $F$, with $\nabla P_2$ being locally Lipschitz on ${\cal N}$.
Let $\{\bm x^k\}$ be the sequence generated by {\rm pDCA$_e$} for solving \eqref{problem}.
Then $\{\bm x^k\}$ converges to a stationary point of $F$.
\end{theorem}
In addition, under the assumptions of Theorem~\ref{thm12} and by assuming further that $\hat E$ is a KL function with exponent $\alpha\in [0,1)$, local convergence rate of the sequence generated by pDCA$_e$ can be characterized by $\alpha$; see \cite[Theorem~4.3]{Wen16}.

The results on global sequential convergence and local convergence rate from \cite{Wen16} mentioned above were derived based on the smoothness assumption on $P_2$. However, as we pointed out in the introduction, the $P_2$ in some important regularizers used in practice are nonsmooth. Moreover, as we shall see in Section~\ref{sec:outlier_detect}, the concave parts in the DC decompositions of many models for simultaneous sparse recovery and outlier detection are typically nonsmooth. Thus, for these problems, global sequential convergence and local convergence rate of the pDCA$_e$ cannot be deduced from \cite{Wen16}. In the next section, we refine the convergence analysis in \cite{Wen16} and establish global convergence of the sequence generated by pDCA$_e$ {\em without} requiring $P_2$ to be smooth.

\section{Convergence analysis}\label{sec3}

In this section, we present our global convergence results for pDCA$_e$. Unlike the analysis in \cite{Wen16}, we do not require $P_2$ to be smooth. The key departure of our analysis from that in \cite{Wen16} is that, instead of using the function $\hat E$ in \eqref{hatE}, we make use of the following auxiliary function and its KL property extensively in our analysis:
\begin{equation}\label{merit_fun}
  E(\bm x, \bm y, \bm w) = f(\bm x) + P_1(\bm x) - \langle\bm x, \bm y\rangle + P_2^*(\bm y) + \frac L2\|\bm x - \bm w\|^2.
\end{equation}
It is easy to see from Young's inequality that
\begin{equation}\label{E_lbd}
  E(\bm x, \bm y, \bm w)\ge f(\bm x) + P_1(\bm x) - P_2(\bm x) + \frac L2\|\bm x - \bm w\|^2 = \hat E(\bm x,\bm w)\ge  F(\bm x).
\end{equation}
Hence, the function $E$ is a majorant of both $\hat E$ and $F$. Similar to the development in \cite[Section~4.1]{Wen16}, we first present some useful properties of $E$ along the sequences $\{\bm x^k\}$ and $\{\bm\xi^k\}$ generated by {\rm pDCA$_e$} in the next proposition. %Recall that for the sequences $\{\bm x^k\}$ and $\{\bm\xi^k\}$ generated by {\rm pDCA$_e$} for solving \eqref{problem}, we denote the set of accumulation points of $\{(\bm x^k, \bm\xi^k, \bm x^{k-1})\}$ by $\Upsilon$.

\begin{proposition}\label{prop1}
Suppose that $F$ in \eqref{problem} is level-bounded and let $E$ be defined in \eqref{merit_fun}. Let $\{\bm x^k\}$ and $\{\bm\xi^k\}$ be the sequences generated by {\rm pDCA$_e$} for solving \eqref{problem}. Then the following statements hold.
\begin{enumerate}[{\rm (i)}]
  \item For any $k\ge 1$,
  \begin{equation}\label{E_dec}
    E(\bm x^{k+1}, \bm\xi^{k+1}, \bm x^k)\le E(\bm x^k, \bm\xi^k, \bm x^{k-1}) - \frac L2(1 - \beta_k^2)\|\bm x^k - \bm x^{k-1}\|^2.
  \end{equation}
  \item The sequences $\{\bm x^k\}$ and $\{\bm\xi^k\}$ are bounded. Hence, the set of accumulation points of the sequence $\{(\bm x^k, \bm\xi^k, \bm x^{k-1})\}$, denoted by $\Upsilon$, is a nonempty compact set.
  \item The limit $\lim_{k\rightarrow\infty}E(\bm x^k, \bm\xi^k, \bm x^{k-1})=: \zeta$ exists and $E\equiv\zeta$ on $\Upsilon$.
  \item There exists $D>0$ so that for any $k\ge 1$, we have
  \begin{equation}\label{eq2}
  \dist(\bm 0, \partial E(\bm x^k, \bm\xi^k, \bm x^{k-1}))\le D(\|\bm x^k - \bm x^{k-1}\| + \|\bm x^{k-1} - \bm x^{k-2}\|).
  \end{equation}
\end{enumerate}
\end{proposition}

\begin{proof}
We first prove (i). Using the definition of $\bm x^{k+1}$ in \eqref{iter_update} as a global minimizer of a strongly convex function, we have
\begin{equation*}
\begin{split}
 & \langle\nabla f(\bm u^k) - \bm\xi^{k+1}, \bm x^{k+1}\rangle + \frac L2\|\bm x^{k+1} - \bm u^k\|^2 + P_1(\bm x^{k+1}) \\
 & \le  \langle\nabla f(\bm u^k) - \bm\xi^{k+1}, \bm x^k\rangle + \frac L2\|\bm x^k - \bm u^k\|^2 + P_1(\bm x^k) - \frac{L}2\|\bm x^{k+1} - \bm x^k\|^2.
\end{split}
\end{equation*}
Rearranging terms in the above inequality, we see that
\[
\begin{split}
  P_1(\bm x^{k+1})& \le \langle\nabla f(\bm u^k) - \bm\xi^{k+1}, \bm x^k - \bm x^{k+1}\rangle + \frac L2\|\bm x^k - \bm u^k\|^2 + P_1(\bm x^k) \\
  &\ \ \ \ - \frac{L}2\|\bm x^{k+1} - \bm x^k\|^2- \frac L2\|\bm x^{k+1} - \bm u^k\|^2.
\end{split}
\]
Using this inequality, we obtain
\begin{align}\label{eq9}
\begin{split}
& E(\bm x^{k+1}, \bm\xi^{k+1}, \bm x^k) = f(\bm x^{k+1}) + P_1(\bm x^{k+1}) - \langle\bm x^{k+1}, \bm\xi^{k+1}\rangle + P_2^*(\bm\xi^{k+1}) + \frac{L}2\|\bm x^{k+1} - \bm x^k\|^2\\
& \le f(\bm x^{k+1}) + \langle\nabla f(\bm u^k) - \bm\xi^{k+1}, \bm x^k - \bm x^{k+1}\rangle + \frac L2\|\bm x^k - \bm u^k\|^2 + P_1(\bm x^k) - \frac{L}2\|\bm x^{k+1} - \bm x^k\|^2\\
&\ \ \ - \frac L2\|\bm x^{k+1} - \bm u^k\|^2 - \langle\bm x^{k+1}, \bm\xi^{k+1}\rangle + P_2^*(\bm\xi^{k+1}) + \frac{L}2\|\bm x^{k+1} - \bm x^k\|^2\\
& = f(\bm x^{k+1}) + \langle\nabla f(\bm u^k) - \bm\xi^{k+1}, \bm x^k - \bm x^{k+1}\rangle + \frac L2\|\bm x^k - \bm u^k\|^2 + P_1(\bm x^k) - \frac L2\|\bm x^{k+1} - \bm u^k\|^2\\
&\ \ + \langle\bm x ^k - \bm x^{k+1}, \bm\xi^{k+1}\rangle - P_2(\bm x^k),\\
\end{split}
\end{align}
where the second equality follows from $\bm\xi^{k+1}\in\partial P_2(\bm x^k)$. Now, using the Lipschitz continuity of $\nabla f$, we see that
\[
f(\bm x^{k+1})\le f(\bm u^k) + \langle\nabla f(\bm u^k), \bm x^{k+1} - \bm u^k\rangle + \frac L2\|\bm x^{k+1} - \bm u^k\|^2.
\]
Combining this with \eqref{eq9}, we deduce further that for $k\ge 1$,
\begin{align}
& E(\bm x^{k+1}, \bm\xi^{k+1}, \bm x^k)\nn\\
& \le f(\bm u^k) + \langle\nabla f(\bm u^k), \bm x^{k+1} - \bm u^k\rangle + \frac L2\|\bm x^{k+1} - \bm u^k\|^2 + \langle\nabla f(\bm u^k) - \bm\xi^{k+1}, \bm x^k - \bm x^{k+1}\rangle\nn\\
&\ \ + \frac L2\|\bm x^k - \bm u^k\|^2 + P_1(\bm x^k) - \frac L2\|\bm x^{k+1} - \bm u^k\|^2 + \langle\bm x ^k - \bm x^{k+1}, \bm\xi^{k+1}\rangle - P_2(\bm x^k)\nn\\
& = f(\bm u^k) + \langle\nabla f(\bm u^k), \bm x^k - \bm u^k\rangle + \frac L2\|\bm x^k - \bm u^k\|^2 + P_1(\bm x^k) - P_2(\bm x^k)\nn\\
& \le f(\bm x^k) + \frac L2\|\bm x^k - \bm u^k\|^2 + P_1(\bm x^k) - \langle\bm x^k, \bm\xi^k\rangle + P_2^*(\bm\xi^k)\nn\\
& = f(\bm x^k) + P_1(\bm x^k) - \langle\bm x^k, \bm\xi^k\rangle + P_2^*(\bm\xi^k) + \frac L2\beta_k^2\|\bm x^k - \bm x^{k-1}\|^2\nn\\
& = E(\bm x^k, \bm\xi^k, \bm x^{k-1}) - \frac L2(1 - \beta_k^2)\|\bm x^k - \bm x^{k-1}\|^2,\nn
\end{align}
where the second inequality follows from the convexity of $f$ and the Young's inequality applied to $P_2$, and the second equality follows from the definition of $\bm u^k$ in \eqref{iter_update}. This proves (i).

For statement (ii), we first note from Theorem~\ref{thm11}(ii) that $\{\bm x^k\}$ is bounded. The boundedness of $\{\bm \xi^k\}$ then follows immediately from this, the continuity of $P_2$ and the fact that $\bm \xi^k \in \partial P_2(\bm x^{k-1})$ for $k\ge 1$. This proves (ii).

Now we prove (iii). First, we see from $\sup_k\beta_k < 1$ and \eqref{E_dec} that the sequence $\{ E(\bm x^k, \bm\xi^k, \bm x^{k-1})\}$ is nonincreasing. Moreover, note from \eqref{E_lbd} and the level-boundedness of $F$ that this sequence is also bounded below. Thus, $\zeta:=\lim_{k\rightarrow\infty}E(\bm x^k, \bm\xi^k, \bm x^{k-1})$ exists.

We next show that $E\equiv \zeta$ on $\Upsilon$. To this end, take any $(\hat{\bm x}, \hat{\bm\xi}, \hat{\bm w})\in\Upsilon$. Then there exists a convergent subsequence $\{(\bm x^{k_i}, \bm\xi^{k_i}, \bm x^{k_i-1})\}$ such that $\lim_{i\rightarrow\infty}(\bm x^{k_i}, \bm\xi^{k_i}, \bm x^{k_i-1}) = (\hat{\bm x}, \hat{\bm\xi}, \hat{\bm w})$. %First, it is easy to see that $\hat{\bm x} = \hat{\bm w}$ because $\|\bm x^{k_i+1} - \bm x^{k_i}\|\rightarrow 0$ according to Theorem~\ref{thm11}(i).
Now, using the definition of $\bm x^{k_i}$ as the minimizer of the $\bm x$-subproblem in \eqref{iter_update}, we have
\begin{eqnarray*}
 & & \langle\nabla f(\bm u^{k_i-1}) - \bm\xi^{k_i}, \bm x^{k_i}\rangle + \frac L2\|\bm x^{k_i} - \bm u^{k_i-1}\|^2 + P_1(\bm x^{k_i})\\
 &\le& \langle\nabla f(\bm u^{k_i-1}) - \bm\xi^{k_i}, \hat{\bm x}\rangle + \frac L2\|\hat{\bm x}- \bm u^{k_i-1}\|^2 + P_1(\hat{\bm x}).
\end{eqnarray*}
Rearranging terms in the above inequality, we obtain further that
\begin{equation}\label{min_def1}
  \langle\nabla f(\bm u^{k_i-1}) - \bm\xi^{k_i}, \bm x^{k_i} - \hat{\bm x}\rangle + \frac L2\|\bm x^{k_i} - \bm u^{k_i-1}\|^2 + P_1(\bm x^{k_i}) \le \frac L2\|\hat{\bm x}- \bm u^{k_i-1}\|^2 + P_1(\hat{\bm x}).
\end{equation}
On the other hand, using the triangle inequality and the definition of $\bm u^k$ in \eqref{iter_update}, we see that
\begin{gather*}
  \|\hat{\bm x} - \bm u^{k_i-1}\| = \|\hat{\bm x} - \bm x^{k_i} + \bm x^{k_i} - \bm u^{k_i-1}\| \le \|\hat{\bm x} - \bm x^{k_i}\| + \|\bm x^{k_i} - \bm u^{k_i-1}\|, \\
  \|\bm x^{k_i} - \bm u^{k_i-1}\| = \|\bm x^{k_i} - \bm x^{k_i-1} - \beta_{k_i-1}(\bm x^{k_i-1} - \bm x^{k_i-2})\| \le \|\bm x^{k_i} - \bm x^{k_i-1}\| + \|\bm x^{k_i-1} - \bm x^{k_i-2}\|.
\end{gather*}
These together with $\lim_{k\rightarrow\infty}\|\bm x^{k+1} - \bm x^k\| = 0$ from Theorem~\ref{thm11}(i) imply
\begin{equation}\label{rel1}
  \|\hat{\bm x} - \bm u^{k_i-1}\|\rightarrow 0 \ \ \mbox{and} \ \ \|\bm x^{k_i} - \bm u^{k_i-1}\|\rightarrow  0.
\end{equation}
In addition, using the continuous differentiability of $f$ and the boundedness of $\{\bm\xi^{k_i}\}$ and $\{\bm u^{k_i}\}$, we have
\begin{equation}\label{rel2}
\lim_{i\to \infty}\langle\nabla f(\bm u^{k_i-1}) - \bm\xi^{k_i}, \bm x^{k_i} - \hat{\bm x}\rangle = 0.
\end{equation}
Thus, we have
\begin{equation*}
\begin{split}
  &\zeta = \lim_{i\rightarrow\infty}E(\bm x^{k_i}, \bm\xi^{k_i}, \bm x^{k_i-1}) = \lim_{i\rightarrow\infty}f(\bm x^{k_i}) + P_1(\bm x^{k_i}) - \langle\bm x^{k_i}, \bm\xi^{k_i}\rangle + P_2^*(\bm\xi^{k_i}) + \frac{L}2\|\bm x^{k_i} - \bm x^{k_i-1}\|^2\\
   &= \lim_{i\rightarrow\infty}f(\bm x^{k_i}) + P_1(\bm x^{k_i}) - \langle\bm x^{k_i}, \bm\xi^{k_i}\rangle + P_2^*(\bm\xi^{k_i}) + \frac{L}2\|\bm x^{k_i} - \bm u^{k_i-1}\|^2\\
   &= \lim_{i\rightarrow\infty}f(\bm x^{k_i}) + \langle\nabla f(\bm u^{k_i-1}) - \bm\xi^{k_i}, \bm x^{k_i} - \hat{\bm x}\rangle + \frac L2\|\bm x^{k_i} - \bm u^{k_i-1}\|^2 +  P_1(\bm x^{k_i}) - \langle\bm x^{k_i}, \bm\xi^{k_i}\rangle + P_2^*(\bm\xi^{k_i})\\
   &\le \limsup_{i\rightarrow\infty}f(\bm x^{k_i}) + \frac L2\|\hat{\bm x}- \bm u^{k_i-1}\|^2 + P_1(\hat{\bm x}) - \langle\bm x^{k_i}, \bm\xi^{k_i}\rangle + P_2^*(\bm\xi^{k_i})\\
   &= \limsup_{i\rightarrow\infty}f(\bm x^{k_i}) + P_1(\hat{\bm x}) - \langle\bm x^{k_i} - \bm x^{k_i-1}, \bm\xi^{k_i}\rangle - P_2(\bm x^{k_i-1})\\
   &= f(\hat{\bm x}) + P_1(\hat{\bm x}) - P_2(\hat{\bm x}) = F(\hat{\bm x})\le  E(\hat{\bm x}, \hat{\bm\xi}, \hat{\bm w}),
\end{split}
\end{equation*}
where the third equality follows from \eqref{rel1} and $\|\bm x^{k_i} - \bm x^{k_i-1}\|\rightarrow0$, the fourth equality follows from \eqref{rel2}, the first inequality follows from \eqref{min_def1}, the fifth equality follows from $\|\hat{\bm x} - \bm u^{k_i-1}\|\rightarrow 0$ (see \eqref{rel1}) and $\bm\xi^{k_i}\in\partial P_2(\bm x^{k_i-1})$, the sixth equality follows from $\|\bm x^{k_i} - \bm x^{k_i-1}\|\rightarrow0$, the boundedness of $\{\bm \xi^k\}$ and the continuity of $f$ and $P_2$, and
the last inequality follows from \eqref{E_lbd}. Finally, since $E$ is lower semicontinuous, we also have
\begin{equation*}
  E(\hat{\bm x}, \hat{\bm\xi}, \hat{\bm w}) \le\liminf_{i\rightarrow\infty}E(\bm x^{k_i}, \bm\xi^{k_i}, \bm x^{k_i-1}) = \zeta.
\end{equation*}
Consequently, $E(\hat{\bm x}, \hat{\bm\xi}, \hat{\bm w}) = \zeta $. We then conclude that $E\equiv\zeta$ from the arbitrariness of $(\hat{\bm x}, \hat{\bm \xi}, \hat{\bm w})$ on $\Upsilon$. This proves (iii).

Finally, we prove (iv). Note that the subdifferential of the function $E$ at the point $(\bm x^k, \bm\xi^k, \bm x^{k-1})$, $k\ge 1$, is:
\begin{equation*}
  \partial E(\bm x^k, \bm\xi^k, \bm x^{k-1}) = \begin{bmatrix}\nabla f(\bm x^k) + \partial P_1(\bm x^k) - \bm\xi^k + L(\bm x^k - \bm x^{k-1})\\-\bm x^k + \partial P_2^*(\bm\xi^k)\\-L(\bm x^k - \bm x^{k-1})\end{bmatrix}.
\end{equation*}
On the other hand, one can see from pDCA$_e$ that $\bm x^{k-1}\in\partial P_2^*(\bm\xi^k)$ for $k\ge 1$. Moreover, we know from the $\bm x$-update in \eqref{iter_update} that for $k\ge 1$,
\begin{equation*}
 - \nabla f(\bm u^{k-1}) + \bm\xi^{k} - L(\bm x^k - \bm u^{k-1})\in\partial P_1(\bm x^k).
\end{equation*}
Using these relations, we see further that for $k\ge 1$,
\begin{equation*}
  \begin{bmatrix}\nabla f(\bm x^k) - \nabla f(\bm u^{k-1}) - L(\bm x^{k-1} - \bm u^{k-1})\\ \bm x^{k-1} - \bm x^k\\-L(\bm x^k - \bm x^{k-1}) \end{bmatrix}\in\partial E(\bm x^k, \bm\xi^k, \bm x^{k-1}).
\end{equation*}
This together with the definition of $\bm u^k$ and the Lipschitz continuity of $\nabla f$ implies that \eqref{eq2} holds for some $D > 0$. This completes the proof.
\end{proof}

Equipped with the properties of $E$ established in Proposition~\ref{prop1}, we are now ready to present our global convergence analysis. We will show that the sequence $\{\bm x^k\}$ generated by pDCA$_e$ is convergent to a stationary point of $F$ in \eqref{problem} under the additional assumption that the function $E$ defined in \eqref{merit_fun} is a KL function. Unlike \cite{Wen16}, we do not impose smoothness assumptions on $P_2$. The line of arguments we use in the proof are standard among convergence analysis based on KL property; see, for example, \cite{At09,At10,At13}. We include the proof for completeness.

\begin{theorem}\label{thm2}
Suppose that $F$ in \eqref{problem} is level-bounded and the $E$ defined in \eqref{merit_fun} is a KL function. Let $\{\bm x^k\}$ be the sequence generated by {\rm pDCA$_e$} for solving \eqref{problem}. Then the sequence $\{\bm x^k\}$ is convergent to a stationary point of $F$. Moreover, $\sum_{k=1}^{\infty}\|\bm x^k - \bm x^{k-1}\| < \infty$.
\end{theorem}

\begin{proof}
In view of Theorem~\ref{thm11}(ii), it suffices to prove that $\{\bm x^k\}$ is convergent and $\sum_{k=1}^{\infty}\|\bm x^k - \bm x^{k-1}\| < \infty$. To this end, we first recall from Proposition~\ref{prop1}(iii) and \eqref{E_dec} that the sequence $\{E(\bm x^k, \bm\xi^k, \bm x^{k-1})\}$ is nonincreasing (Recall that $\sup_k \beta_k < 1$) and $\zeta=\lim_{k\rightarrow\infty}E(\bm x^k, \bm\xi^k, \bm x^{k-1})$ exists. Thus, if there exists some $N > 0$ such that $E(\bm x^N, \bm\xi^N, \bm x^{N-1}) = \zeta$, then it must hold that $E(\bm x^k, \bm\xi^k, \bm x^{k-1}) = \zeta$ for all $k\ge N$. Therefore, we know from \eqref{E_dec} that $\bm x^k = \bm x^N$ for any $k\ge N$, implying that $\{\bm x^k\}$ converges finitely.

We next consider the case that $E(\bm x^k, \bm\xi^k, \bm x^{k-1})>\zeta$ for all $k$. Recall from Proposition~\ref{prop1}(ii) that $\Upsilon$ is the (compact) set of accumulation points of $\{(\bm x^k,\bm \xi^k,\bm x^{k-1})\}$. Since $E$ satisfies the KL property at each point in the compact set $\Upsilon\subseteq {\rm dom}\,\partial E$ and $E\equiv\zeta$ on $\Upsilon$, by Lemma~\ref{lemma_KL}, there exist an $\epsilon > 0$ and a continuous concave function $\varphi\in\Xi_a$ with $a > 0$ such that
\begin{equation}\label{KL_Enew}
  \varphi'(E(\bm x, \bm y, \bm w) - \zeta)\cdot\dist(\bm 0, \partial E(\bm x, \bm y, \bm w))\ge 1
\end{equation}
for all $(\bm x, \bm y, \bm w)\in U$, where
\begin{equation*}
  U = \{(\bm x, \bm y, \bm w): \dist((\bm x, \bm y, \bm w),\Upsilon) <\epsilon\}\cap\{(\bm x, \bm y, \bm w): \zeta < E(\bm x, \bm y, \bm w) < \zeta + a\}.
\end{equation*}
Since $\Upsilon$ is the set of accumulation points of the bounded sequence $\{(\bm x^k, \bm\xi^k, \bm x^{k-1})\}$, we have
\begin{equation*}
  \lim_{k\rightarrow\infty}\dist((\bm x^k, \bm\xi^k, \bm x^{k-1}),\Upsilon) = 0.
\end{equation*}
Hence, there exists $N_1 > 0$ such that $\dist((\bm x^k, \bm\xi^k, \bm x^{k-1}),\Upsilon) < \epsilon$ for any $k\ge N_1$. In addition, since the sequence $\{E(\bm x^k, \bm\xi^k, \bm x^{k-1})\}$ converges to $\zeta$ by Proposition~\ref{prop1}(iii), there exists $N_2 > 0$ such that $\zeta < E(\bm x^k, \bm\xi^k, \bm x^{k-1})< \zeta + a$ for any $k\ge N_2$. Let $\bar{N}=\max\{N_1,N_2\}$. Then the sequence $\{(\bm x^k, \bm\xi^k, \bm x^{k-1})\}_{k\ge\bar{N}}$ belongs to $U$ and we deduce from \eqref{KL_Enew} that
\begin{equation}\label{KL_E1_new}
 \varphi'(E(\bm x^k, \bm\xi^k, \bm x^{k-1}) - \zeta)\cdot\dist(\bm 0, \partial E(\bm x^k, \bm\xi^k, \bm x^{k-1}))\ge 1,\ \ \forall~k\ge\bar{N}.
\end{equation}
Using the concavity of $\varphi$, we see further that for any $k\ge\bar{N}$,
\begin{equation*}
  \begin{split}
      & [\varphi(E(\bm x^k, \bm\xi^k, \bm x^{k-1}) - \zeta)- \varphi(E(\bm x^{k+1}, \bm\xi^{k+1}, \bm x^k) - \zeta)]\cdot\dist(\bm 0, \partial E(\bm x^k, \bm\xi^k, \bm x^{k-1})) \\
       & \ge  \varphi'(E(\bm x^k, \bm\xi^k, \bm x^{k-1}) - \zeta)\cdot\dist(\bm 0, \partial E(\bm x^k, \bm\xi^k, \bm x^{k-1}))\cdot(E(\bm x^k, \bm\xi^k, \bm x^{k-1}) - E(\bm x^{k+1}, \bm\xi^{k+1}, \bm x^k))\\
&\ge E(\bm x^k, \bm\xi^k, \bm x^{k-1}) - E(\bm x^{k+1}, \bm\xi^{k+1}, \bm x^k),
   \end{split}
\end{equation*}
where the last inequality follows from \eqref{KL_E1_new} and \eqref{E_dec}, which states that the sequence $\{E(\bm x^k, \bm\xi^k, \bm x^{k-1})\}$ is nonincreasing, thanks to $\sup_k \beta_k < 1$. Combining this with \eqref{E_dec} and \eqref{eq2}, and writing $\Delta_k:= \varphi(E(\bm x^k, \bm\xi^k, \bm x^{k-1}) - \zeta)- \varphi(E(\bm x^{k+1}, \bm\xi^{k+1}, \bm x^k) - \zeta)$ and $C:= \frac L2(1-\sup_k\beta_k^2) > 0$, we have for any $k\ge\bar{N}$ that,
\begin{equation*}
  \|\bm x^k - \bm x^{k-1}\|^2\le\frac{D}{C}\Delta_k(\|\bm x^k - \bm x^{k-1}\| + \|\bm x^{k-1} - \bm x^{k-2}\|).
\end{equation*}
Therefore, applying the arithmetic mean-geometric mean (AM-GM) inequality, we obtain
\begin{eqnarray*}
   \|\bm x^k - \bm x^{k-1}\| &\le& \sqrt{\frac{2D}{C}\Delta_k}\cdot\sqrt{\frac{\|\bm x^k - \bm x^{k-1}\| + \|\bm x^{k-1} - \bm x^{k-2}\|}{2}} \\
   &\le& \frac{D}{C}\Delta_k + \frac14\|\bm x^k - \bm x^{k-1}\| + \frac14\|\bm x^{k-1} - \bm x^{k-2}\|,
\end{eqnarray*}
which implies that
\begin{equation}\label{sum_ineq1}
  \frac12\|\bm x^k - \bm x^{k-1}\|\le\frac{D}{C}\Delta_k + \frac14(\|\bm x^{k-1} - \bm x^{k-2}\| - \|\bm x^k - \bm x^{k-1}\|).
\end{equation}
Summing both sides of \eqref{sum_ineq1} from $k=\bar{N}$ to $\infty$ and noting that $\sum_{k=\bar{N}}^{\infty}\Delta_k\le \varphi(E(\bm x^{\bar{N}}, \bm\xi^{\bar{N}}, \bm x^{\bar{N}-1}) - \zeta)$, we obtain
\begin{equation*}
  \sum_{k=\bar{N}}^{\infty}\|\bm x^k - \bm x^{k-1}\|\le\frac{2D}{C}\varphi(E(\bm x^{\bar{N}}, \bm\xi^{\bar{N}}, \bm x^{\bar{N}-1}) - \zeta) + \frac12\|\bm x^{\bar{N}-1} - \bm x^{\bar{N}-2}\|<\infty,
\end{equation*}
which implies that the sequence $\{\bm x^k\}$ is convergent. This completes the proof.
\end{proof}

Before closing this section, we would like to point out that, similar to the analysis in \cite[Theorem~3.4]{At10}, one can also establish local convergence rate of the sequence $\{\bm x^k\}$ under the assumptions in Theorem~\ref{thm2} and the additional assumption that the function $E$ defined in \eqref{merit_fun} is a KL function with exponent $\alpha\in[0,1)$. As an illustration, suppose that the exponent is $\frac12$. Then one can show that $\{\bm x^k\}$ converges locally linearly to a stationary point of $F$ in \eqref{problem}. Indeed, according to Proposition~\ref{prop1}, it holds that $E \equiv \zeta$ for some constant $\zeta$ on the compact set $\Upsilon$. Using this, the KL assumption on $E$ and following the proof of \cite[Lemma~6]{Bo14}, we conclude that the uniform KL property in \eqref{kl_defii} holds for $E$ (in place of $h$) with $\Gamma = \Upsilon$ and $\varphi(s) = cs^\frac12$ for some $c > 0$. Now, proceed as in the proof of Theorem~\ref{thm2} and use $\varphi(s) = cs^\frac12$ in \eqref{KL_E1_new}, we have
\begin{equation*}
 \sqrt{E(\bm x^k, \bm\xi^k, \bm x^{k-1}) - \zeta} \le \frac{c}2\cdot\dist\left(0, \partial E(\bm x^k, \bm\xi^k, \bm x^{k-1})\right).
\end{equation*}
Combining this with \eqref{eq2} and \eqref{E_dec}, we see further that for all sufficiently large $k$,
\[
\begin{split}
  & E(\bm x^{k+1}, \bm\xi^{k+1}, \bm x^{k}) - \zeta \le E(\bm x^k, \bm\xi^k, \bm x^{k-1}) - \zeta \le \frac{c^2}4\dist^2\left(0, \partial E(\bm x^k, \bm\xi^k, \bm x^{k-1})\right)\\
  & \le \frac{c^2D^2}4(\|\bm x^k - \bm x^{k-1}\| + \|\bm x^{k-1} - \bm x^{k-2}\|)^2\le \frac{c^2D^2}2(\|\bm x^k - \bm x^{k-1}\|^2 + \|\bm x^{k-1} - \bm x^{k-2}\|^2)\\
  & \le C (E(\bm x^{k-1}, \bm\xi^{k-1}, \bm x^{k-2}) - E(\bm x^{k+1}, \bm\xi^{k+1}, \bm x^{k})),
\end{split}
\]
where $C = \frac{c^2D^2}{L(1-\sup_k\beta_k^2)}$. One can then deduce that the sequence $\{E(\bm x^{k+1}, \bm\xi^{k+1}, \bm x^{k})\}$ is $R$-linearly convergent from the above inequality. The $R$-linear convergence of $\{\bm x^k\}$ now follows from this and \eqref{E_dec}.

Thus, the KL exponent of $E$ plays an important role in analyzing the local convergence rate of the pDCA$_e$. This is to be contrasted with \cite[Theorem~4.3]{Wen16}, which analyzed local convergence of the pDCA$_e$ based on the KL exponent of the function $\hat E$ in \eqref{hatE} under an additional smoothness assumption on $P_2$. In the next section, we study a relationship between the KL assumption on $E$ and that on $\hat E$; the latter was used in the convergence analysis in \cite{Wen16}.

%\begin{corollary}\label{kl_thm}
%Let $\{\bm x^k\}$ be a sequence generated by {\rm pDCA$_e$} for solving \eqref{problem} and $E$ be defined in \eqref{merit_fun}. Suppose that $E$ is a KL function with exponent $\frac12$. Then $\{\bm x^k\}$ converges locally linearly to a stationary point of \eqref{problem}.
%\end{corollary}
%\begin{proof}
%We first see from Proposition~\ref{prop1}(i) and (iii) that there exists some $c_1 > 0$ such that
%\begin{equation*}
%  \|\bm x^k - \bm x^{k-1}\| \le c_1\cdot\sqrt{E(\bm x^k, \bm\xi^k, \bm x^{k-1}) - E(\bm x^{k+1}, \bm\xi^{k+1}, \bm x^k)} \le c_1\cdot\sqrt{E(\bm x^k, \bm\xi^k, \bm x^{k-1}) - \zeta}.
%\end{equation*}
%On the other hand, since $E$ is a KL function with exponent $\frac12$, for any fixed $(\bar{\bm x}, \bar{\bm\xi}, \bar{\bm x})\in\Upsilon$, there exists some $c_2 >0$ such that
%\begin{equation*}
% \sqrt{E(\bm x^k, \bm\xi^k, \bm x^{k-1}) - \zeta} = \sqrt{E(\bm x^k, \bm\xi^k, \bm x^{k-1}) - E(\bar{\bm x}, \bar{\bm\xi}, \bar{\bm x})} \le c_2\cdot\dist\left(0, \partial E(\bm x^k, \bm\xi^k, \bm x^{k-1})\right).
%\end{equation*}
%Thes together with \eqref{eq2} imply that
%\begin{equation*}
%  \|\bm x^k - \bm x^{k-1}\| \le c_1c_2D\cdot(\|\bm x^k - \bm x^{k-1}\| + \|\bm x^{k-1} - \bm x^{k-2}\|).
%\end{equation*}
%\end{proof}

\section{Connecting various KL assumptions}\label{sec4}

As discussed at the end of the previous section, the convergence rate of pDCA$_e$ can be analyzed based on the KL exponent of the function $E$ in \eqref{merit_fun}. Specifically, when $F$ in \eqref{problem} is level-bounded, the sequence $\{\bm x^k\}$ generated by pDCA$_e$ is locally linearly convergent if the exponent is $\frac12$. On the other hand, when $P_2$ satisfies a certain smoothness assumption (see the assumptions on $P_2$ in Theorem~\ref{thm12}), a local linear convergence result was established in \cite[Theorem~4.3]{Wen16} under a different KL assumption: by assuming that the function $\hat E$ in \eqref{hatE} is a KL function with exponent $\frac12$. In this section, we study a relationship between these two KL assumptions.

%We revisit the formulation of $E$ in \eqref{merit_fun}:
%\begin{equation*}
%  E(\bm x,\bm y,\bm z,\bm w) = \frac12\|\bm A\bm x-\bm b - \bm z\|^2 + \delta_{\Omega}(\bm z) + P_1(\bm x) - \langle\bm x, \bm y\rangle + P_2^*(\bm y) + \frac{L}2\|\bm x - \bm w\|^2,
%\end{equation*}
We first prove the following theorem, which studies the KL exponent of a majorant formed from the original function by majorizing the concave part.

\begin{theorem}\label{conj}
Let $h(\bm x) = Q_1(\bm x) - Q_2(\bm A\bm x)$, where $Q_1$ is proper closed, $Q_2$ is convex with globally Lipschitz gradient and $\bm A$ is a linear mapping. Suppose that $h$ satisfies the KL property at $\bar{\bm x}\in {\rm dom}\,\partial h$ with exponent $\frac12$. Then $H(\bm x, \bm y) = Q_1(\bm x) - \langle \bm A\bm x, \bm y\rangle + Q_2^*(\bm y)$ satisfies the KL property at $(\bar{\bm x},\nabla Q_2(A\bar{\bm x}))\in {\rm dom}\,\partial H$ with exponent $\frac12$.
\end{theorem}
\begin{proof}
  Note that it is routine to prove that $(\bar{\bm x},\nabla Q_2(\bm A\bar{\bm x}))\in {\rm dom}\,\partial H$ and that
  \begin{equation}\label{Hequal}
  H(\bar{\bm x},\nabla Q_2(\bm A\bar{\bm x})) = h(\bar{\bm x}).
  \end{equation}
  For any $(\bm x,\bm y)\in {\rm dom}\,\partial H$, let $\tilde{\bm u}\in \partial Q_2^*(\bm y)$. Then we have
  \begin{equation}\label{eeq1}
    \begin{split}
      &H(\bm x,\bm y) = h(\bm x) + Q_2(\bm A\bm x) + Q_2^*(\bm y) - \langle \bm A\bm x,\bm y\rangle\\
      & = h(\bm x) + Q_2(\bm A\bm x) - Q_2(\tilde{\bm u}) + Q_2(\tilde{\bm u}) + Q_2^*(\bm y) - \langle \tilde{\bm u},\bm y\rangle + \langle \tilde{\bm u}- \bm A\bm x,\bm y\rangle\\
      & = h(\bm x) + Q_2(\bm A\bm x) - Q_2(\tilde{\bm u}) + \langle \tilde{\bm u}- \bm A\bm x,\bm y\rangle \le h(\bm x) + \frac\ell{2}\|\tilde{\bm u} - \bm A\bm x\|^2,
    \end{split}
  \end{equation}
  where the last equality follows from the fact that $\tilde{\bm u}\in \partial Q_2^*(\bm y)$, and the inequality follows from $\tilde{\bm u}\in \partial Q_2^*(\bm y)$ (so that $\bm y = \nabla Q_2(\tilde{\bm u})$) and the Lipschitz continuity of $\nabla Q_2$, with $\ell$ being its Lipschitz continuity modulus. Taking infimum over $\tilde{\bm u}\in \partial Q_2^*(\bm y)$, we see from \eqref{eeq1} that
  \begin{equation}\label{eeq2}
  H(\bm x,\bm y) \le h(\bm x) + \frac\ell{2}{\rm dist}^2(\bm A\bm x,\partial Q_2^*(\bm y))
  \end{equation}
  for any $(\bm x,\bm y)\in {\rm dom}\,\partial H$.

  Since $h$ has KL exponent $\frac12$ at $\bar{\bm x}\in {\rm dom}\,\partial h = {\rm dom}\,\partial Q_1$, there exist $c \in (0,1)$ and $\epsilon > 0$ so that
  \begin{equation}\label{eq0}
  {\rm dist}^2(\bm 0,\partial h(\bm x))\ge c(h(\bm x) - h(\bar{\bm x}))
  \end{equation}
  whenever $\|\bm x - \bar{\bm x}\|\le\epsilon$, $\bm x\in {\rm dom}\,\partial Q_1$ and $h(\bm x) < h(\bar{\bm x})+\epsilon$.\footnote{The requirement $h(\bar{\bm x}) < h(\bm x)$ is dropped because \eqref{eq0} holds trivially when $h(\bar{\bm x}) \ge h(\bm x)$.} Moreover, since $Q_2$ has globally Lipschitz gradient, we see that $\partial Q_2^*$ is metrically regular at $(\nabla Q_2(\bm A\bar{\bm x}),\bm A\bar{\bm x})$; see \cite[Theorem~9.43]{RW98}. By shrinking $\epsilon$ and $c$ if necessary, we conclude that
  \begin{equation}\label{eq-1}
  c\|\bm A^\top\bm y - \bm A^\top\nabla Q_2(\bm A\bm x)\|\le c\|\bm A^\top\|\|\bm y - \nabla Q_2(\bm A\bm x)\| \le {\rm dist}(\bm A\bm x,\partial Q^*_2(\bm y))
  \end{equation}
  whenever $\max\{\|\bm x - \bar{\bm x}\|,\|\bm y - \nabla Q_2(\bm A\bar{\bm x})\|\}\le \epsilon$.

  Now, consider any $(\bm x,\bm y)\in {\rm dom}\,\partial H$ satisfying $\max\{\|\bm x - \bar{\bm x}\|,\|\bm y - \nabla Q_2(\bm A\bar{\bm x})\|\}\le \epsilon$ and
  \[
 H(\bar{\bm x},\nabla Q_2(\bm A\bar{\bm x})) \le H(\bm x,\bm y) < H(\bar{\bm x},\nabla Q_2(\bm A\bar{\bm x})) + \epsilon.
  \]
  Then we have for any such $(\bm x,\bm y)$ that
  \begin{equation}\label{eq22}
  \bm x\in {\rm dom}\,\partial Q_1,\ \ \bm y\in {\rm dom}\,\partial Q_2^*\ \ {\rm and}\ \ h(\bar{\bm x}) +\epsilon = H(\bar{\bm x},\nabla Q_2(\bm A\bar{\bm x})) + \epsilon > H(\bm x,\bm y)\ge h(\bm x),
  \end{equation}
  where, in the third relation, the first equality is due to \eqref{Hequal} and the last inequality is a consequence of Young's inequality.
  Furthermore, for any such $(\bm x,\bm y)$, we have
  \begin{equation*}
    \begin{split}
      &3\,{\rm dist}^2(\bm 0,\partial H(\bm x,\bm y)) = 3 \, {\rm dist^2}\left(\bm 0,\begin{bmatrix}
        -\bm A^\top\bm y + \partial Q_1(\bm x)\\ -\bm A\bm x + \partial Q_2^*(\bm y)
      \end{bmatrix}\right)= 3\,{\rm dist}^2(\bm A^\top\bm y,\partial Q_1(\bm x)) + 3\,{\rm dist}^2(\bm A\bm x,\partial Q^*_2(\bm y))\\
      & \ge 2\,{\rm dist}^2(\bm A^\top\bm y,\partial Q_1(\bm x)) + 2\,{\rm dist}^2(\bm A\bm x,\partial Q^*_2(\bm y)) + {\rm dist}^2(\bm A\bm x,\partial Q^*_2(\bm y))\\
      & \ge c^2[2\,{\rm dist}^2(\bm A^\top\bm y,\partial Q_1(\bm x)) + 2\|\bm A^\top\bm y - \bm A^\top\nabla Q_2(\bm A\bm x)\|^2 + {\rm dist}^2(\bm A\bm x,\partial Q^*_2(\bm y))]\\
      & \ge c^2\left[{\rm dist}^2(\bm 0,\partial h(\bm x)) + {\rm dist}^2(\bm A\bm x,\partial Q^*_2(\bm y))\right] \ge c'\left[\frac1{c}{\rm dist}^2(\bm 0,\partial h(\bm x)) + \frac\ell{2}{\rm dist}^2(\bm A\bm x,\partial Q^*_2(\bm y))\right]\\
      & \ge c'\left[h(\bm x) - h(\bar{\bm x}) + \frac\ell{2}{\rm dist}^2(\bm A\bm x,\partial Q_2^*(\bm y))\right] \ge c'\left[H(\bm x,\bm y) - h(\bar{\bm x})\right] = c'\left[H(\bm x,\bm y) - H(\bar{\bm x},\nabla Q_2(\bm A\bar{\bm x}))\right]
    \end{split}
  \end{equation*}
  for $c' := \frac{c^2}{(\frac1c+\frac\ell{2})}$, where the second inequality follows from \eqref{eq-1} and the fact that $c<1$, the third inequality follows from the relation $2(a^2 + b^2) \ge (a+b)^2$ for $a = {\rm dist}(\bm A^\top\bm y,\partial Q_1(\bm x))$ and $b = \|\bm A^\top\bm y - \bm A^\top\nabla Q_2(\bm A\bm x)\|$, the triangle inequality and the definition of $h$, the second last inequality follows from \eqref{eq0} and \eqref{eq22}, while the last inequality follows from \eqref{eeq2}. The last equality is due to \eqref{Hequal}. This completes the proof.
\end{proof}

We are ready to prove the main theorem in this section, which is now an easy corollary of Theorem~\ref{conj}. The first conclusion studies a relationship between the KL assumption used in our analysis and the one used in the analysis in \cite{Wen16}, while the second conclusion shows that one may deduce the KL exponent of the function $E$ in \eqref{merit_fun} directly from that of the original objective function $F$ in \eqref{problem}.

\begin{theorem}\label{KL_imp}
Let $F$, $\hat E$ and $E$ be defined in \eqref{problem}, \eqref{hatE} and \eqref{merit_fun} respectively. Suppose in addition that $P_2$ has globally Lipschitz gradient. Then the following statements hold:
\begin{enumerate}[{\rm (i)}]
  \item If $\hat E$ is a KL function with exponent $\frac12$, then $E$ is a KL function with exponent $\frac12$.
  \item If $F$ is a KL function with exponent $\frac12$, then $E$ is a KL function with exponent $\frac12$.
\end{enumerate}
\end{theorem}

\begin{proof}
We first prove (i). Recall from \cite[Lemma~2.1]{Li_TK_Calculus} that it suffices to prove that $E$ satisfies the KL property with exponent $\frac12$ at all points $(\bm x,\bm y,\bm w)$ verifying ${\bm 0}\in \partial E(\bm x,\bm y,\bm w)$. To this end, let $(\bar {\bm x},\bar {\bm y},\bar {\bm w})$ satisfy ${\bm 0}\in \partial E(\bar {\bm x},\bar {\bm y},\bar {\bm w})$. Then we obtain from the definition of $E$ that
\begin{equation}\label{hahaha}
\bm 0 \in \nabla f(\bar {\bm x}) + \partial P_1(\bar {\bm x}) - \bar {\bm y} + L(\bar {\bm x} - \bar {\bm w}),\ \ \ \bar {\bm x}\in \partial P_2^*(\bar {\bm y}),\ \ \ \bar {\bm x} = \bar {\bm w}.
\end{equation}
Plugging the second and the third relations above into the first relation gives
\[
\bm 0 \in \nabla f(\bar {\bm x}) + \partial P_1(\bar {\bm x}) - \nabla P_2(\bar {\bm x}).
\]
This further implies ${\bm 0}\in \partial \hat E(\bar {\bm x},\bar {\bm x})$, and hence $(\bar {\bm x},\bar {\bm x})\in {\rm dom}\,\partial\hat E$.
Thus, by assumption, the function $\hat E$ satisfies the KL property with exponent $\frac12$ at $(\bar {\bm x},\bar {\bm x})$.
Since
\[
\hat E(\bm x,\bm w) = f(\bm x) + P_1(\bm x) + \frac{L}2\|\bm x - \bm w\|^2 - P_2\left(\begin{bmatrix}
  \bm I & \bm 0
\end{bmatrix}\begin{bmatrix}
  \bm x\\ \bm w
\end{bmatrix}\right),
\]
we conclude immediately from Theorem~\ref{conj} that
\begin{eqnarray*}
  E(\bm x, \bm y, \bm w) = f(\bm x) + P_1(\bm x) - \langle\bm x, \bm y\rangle + P_2^*(\bm y) + \frac L2\|\bm x - \bm w\|^2
\end{eqnarray*}
satisfies the KL property with exponent $\frac12$ at $(\bar {\bm x},\nabla P_2(\bar {\bm x}),\bar {\bm x})$, which is just $(\bar {\bm x},\bar {\bm y}, \bar {\bm w})$ in view of the second and third relations in \eqref{hahaha} and the smoothness of $P_2$. This proves (i).

We now prove (ii). In view of (i) and \cite[Lemma~2.1]{Li_TK_Calculus}, it suffices to show that $\hat E$ satisfies the KL property with exponent $\frac12$ at all points $(\bm x,\bm y)$ verifying ${\bm 0}\in \partial \hat E(\bm x,\bm y)$. To this end, let $(\bar {\bm x}, \bar {\bm w})$ satisfy ${\bm 0}\in \partial \hat E(\bar {\bm x}, \bar {\bm w})$. Then we see from the definition of $\hat E$ that
\begin{equation}\label{hahahaha}
\bm 0 \in \nabla f(\bar {\bm x}) + \partial P_1(\bar {\bm x}) - \nabla P_2(\bar {\bm x}) + L(\bar {\bm x} - \bar {\bm w}),\ \ \ \bar {\bm x} = \bar {\bm w}.
\end{equation}
These relations show that ${\bm 0}\in \partial F(\bar{\bm x})$, and hence $\bar {\bm x}\in {\rm dom}\,\partial F$. This together with the KL assumption on $F$ and \cite[Theorem~3.6]{Li_TK_Calculus} implies that $\hat E$ satisfies the KL property at $(\bar {\bm x}, \bar {\bm x})$, which is just $(\bar {\bm x}, \bar {\bm w})$ in view of the second relation in \eqref{hahahaha}. This completes the proof.
\end{proof}

Before closing this section, we present in the following corollary some specific choices of $F$ in \eqref{problem} whose corresponding function $E$ defined in \eqref{merit_fun} is a KL function with exponent $\frac12$.

\begin{corollary}
Let $F$ and $E$ be defined in \eqref{problem} and \eqref{merit_fun} respectively. Suppose that $f$ is quadratic and $P_1 - P_2$ is the {\rm MCP} or {\rm SCAD} function. Then $E$ is a KL function with exponent $\frac12$.
\end{corollary}

\begin{proof}
Notice from \cite[Table~1]{Gong13} that for the MCP or SCAD function, $P_1$ is a positive multiple of the $\ell_1$ norm and $P_2$ is convex with globally Lipschitz gradient. Thus, by Theorem~\ref{KL_imp}(ii), it suffices to prove that $F$ is a KL function with exponent $\frac12$. Similar to the arguments in \cite[Section 5.2]{Li_TK_Calculus}, using the special structure of the MCP or SCAD function, one can write
\begin{equation*}
  F(\bm x) = f(\bm x) + \sum_{i=1}^n\min_{1\le\ell\le m_i}\left\{f_{i,\ell}(x_i) + \delta_{C_{i,\ell}}(x_i)\right\} = \min_{j\in\mathcal{J}}\left\{f(\bm x) + \sum_{i=1}^n\left[f_{i,j_i}(x_i) + \delta_{C_{i,j_i}}(x_i)\right]\right\},
\end{equation*}
where $C_{i,\ell}$ are closed intervals, $f_{i,\ell}$ are quadratic (or linear) functions, $1\le\ell\le m_i$, $1\le i\le n$, and $\mathcal{J} = \{(j_1,\ldots,j_n)\in\mathbb{N}^n: 1\le j_i\le m_i\ \forall\ i\}$. Notice also that $F$ is a continuous function. Thus, by \cite[Corollary 5.2]{Li_TK_Calculus}, we see that $F$ is a KL function with exponent $\frac12$. This completes the proof.
\end{proof}

\section{Application of pDCA$_e$ to simultaneous sparse recovery and outlier detection}\label{application}

Like the problem of sparse learning/recovery discussed in the introduction, the problem of outlier detection is another classical topic in statistics and signal processing. In particular, it has been extensively studied in the area of machine learning.
In this context, outliers refer to observations that are somehow statistically different from the majority of the training instances. %As we will discuss later, machine learning models that consider outliers can be naturally formulated as nonconvex optimization problems.
On the other hand, in signal processing, outlier detection problems naturally arise when the signals transmitted are contaminated by both Gaussian noise and electromyographic noise: the latter exhibits impulsive behavior and results in extreme measurements/outliers; see, for example, \cite{PLF12}. %As we shall see later, compressed sensing models that take outliers into account can also be formulated as nonconvex optimization problems.

In this section, we will present a nonconvex optimization model incorporating {\em both} sparse learning/recovery and outlier detection, and discuss how it can be solved by the pDCA$_e$.
This is not the first work combining sparse learning/recovery and outlier detection.
For instance, there is a huge literature on {\em robust} compressed sensing, which uses the $\ell_1$ regularizer to identify outliers and recover the underlying sparse signal; see \cite{CRE2016} and the references therein.
As for statistical learning, papers such as \cite{Alfons+etal:13,Hoeting+etal:96,MenjogeWelsch:10} already studied such combined models, but their algorithms
are simple search algorithms through the space of possible feature subsets and/or the space of possible sample subsets. Recently, the papers \cite{Lo:17,SmuclerYohai:17} studied nonconvex-regularized robust regression models based on $M$-estimators. They mainly studied  theoretical properties (e.g., consistency, breakdown point) of the proposed models
and the following disadvantages were left in their algorithms:
 Smucler and Yohai's algorithm  \cite{SmuclerYohai:17} is only for the $\ell_1$ regularizer, and Loh's algorithm, which is based on composite gradient descent \cite{Lo:17}, requires a carefully-chosen initial solution and does not have a global convergence guarantee.

\subsection{Simultaneous sparse recovery and outlier detection}\label{sec:outlier_detect}

In this section, as motivations, we present two concrete scenarios where the problem of simultaneous sparse recovery and outlier detection arises: robust compressed sensing in signal processing and least trimmed squares regression with variable selection in statistics.

\subsubsection{Robust compressed sensing}

In compressed sensing, an original sparse or approximately sparse signal in high dimension is compressed and then transmitted via some channels. The task is to recover the original high dimensional signal from the relatively lower dimensional possibly noisy received signal. This problem is NP hard in general; see \cite{Nat95}.

When there is no noise in the transmission, the recovery problem can be shown to be equivalent to an $\ell_1$ minimization problem under some additional assumptions; see, for example, \cite{CT05,D06}. Recently, various nonconvex models have also been proposed for recovering the underlying sparse/approximately sparse signal; see, for example, \cite{CWB08,C07,CY08}. These models empirically require fewer measurements than their convex counterparts for recovering signals of the same sparsity level.

While the noiseless scenario leads to a theory of exact recovery, in practice, the received signals are noisy. This latter scenario has also been extensively studied in the literature, with Gaussian measurement noise being the typical noise model; see, for example, \cite{CRT06,SCY08}. However, in certain compressed sensing system, the signals can be corrupted by {\em both} Gaussian noise and electromyographic noise: the latter exhibits impulsive behavior and results in extreme measurements/outliers \cite{PLF12}. In the literature, the following model was proposed for handling noise and outliers simultaneously, which makes use of $\ell_1$ regularizer for both sparse recovery and outlier detection; see the recent exposition \cite{CRE2016} and references therein:
\begin{align*}
\begin{array}{cl}
\displaystyle  \min_{\bm x\in\bR^n, \bm z\in\bR^m} & \displaystyle  \tau\|\bm x\|_1 + \|\bm z\|_1 \\
\mbox{ s.t. } & \displaystyle \|\bm A\bm x - \bm b - \bm z\|\le\epsilon;
\end{array}
\end{align*}
here, $\bm A$ is the sensing matrix, $\bm b$ is the possibly noisy received signal, and $\tau > 0$ and $\epsilon > 0$ are parameters controlling the sparsity in $\bm x$ and the allowable noise level, respectively.
In this section, we describe an alternative model that can incorporate some prior knowledge of the number of outliers. In our model, instead of relying on the $\ell_1$ norm for detecting outliers, we employ the $\ell_0$ norm directly, assuming a rough (upper) estimation $r$ of the number of outliers. We also allow the use of possibly nonconvex regularizers $P$ for inducing sparsity in the signal: these kinds of regularizers have been widely used in the literature and have been shown empirically to work better than convex regularizers; see, for example, \cite{CWB08,C07,CY08,FL09,SCY08}.
Specifically, our model takes the following form:
\begin{align}\label{robust_cs}
\begin{array}{cl}
\displaystyle  \min_{\bm x\in\bR^n, \bm z\in\bR^m} & \displaystyle  \frac12\|\bm A\bm x - \bm b- \bm z\|^2 + P(\bm x) \\
\mbox{ s.t. } & \displaystyle \|\bm z\|_0\le r.
\end{array}
\end{align}
Notice that at optimality, at most $r$ number of $z_i$'s will be nonzero and equal to the corresponding $[\bm A\bm x - \bm b]_i$, zeroing out the corresponding terms in the least squares. Thus, once the nonzero entries of $\bm z$ at optimality are identified, the problem reduces to a standard compressed sensing problem with at least $m-r$ measurements: for this class of problem, (approximate) recovery of the original sparse signal is possible if there are sufficient measurements, assuming the sensing matrix is generated according to certain distributions \cite{CT05,CRT06}. This means that one only needs a reasonable {\em upper bound} on the number of outliers so that $m-r$ is not too small in order to recover the original approximately sparse signal, assuming a random sensing matrix and that the outliers are successfully detected. This is in contrast to some $\ell_0$ based approaches such as the iterative hard thresholding (IHT) for compressed sensing \cite{BT09}, where the {\em exact knowledge} of the sparsity level is needed for recovering the signal.

\subsubsection{Least trimmed squares regression with variable selection}\label{sec:5.1.2}

In statistics, suppose that we have data samples $\{\bm a_i, b_i\}_{i=1}^m$, where $\bm a_i \in \bR^n, b_i \in \bR$, and suppose that some data samples are from anomalous observations; such observations are called outliers. These outliers may be caused by mechanical faults, human errors, instrument errors, changes in system behaviour, etc. We need to identify and remove the outliers to improve the prediction performance of the regression model.
In this section, we consider the problem of simultaneously identifying the outliers in the set of samples $\{\bm a_i, b_i\}_{i=1}^m$ and recovering a vector $\bm x^* \in \bR^n$ using
$\{\bm a_i, b_i\}_{i=1}^m$ but the outliers.

Statisticians and data analysts have been searching for regressors
 which are not affected by outliers, i.e., the regressors that are robust with respect to outliers.
 Least trimmed squares (LTS) regression \cite{Rousseeuw83,RousseeuwLeroy87} is popular as a  robust regression model and can be
formulated as a nonlinear mixed zero-one integer optimization problem (see (2.1.1) in \cite{GiloniPadberg02}):
\begin{align*}
\begin{array}{cl}
\displaystyle  \min_{\begin{subarray}
  \ \bm s \in \{0,1\}^m\!\!\!\!,\\
  \ \ \ \bm x\in \bR^n
\end{subarray}} & \displaystyle  \sum_{i=1}^m s_i(\bm a_i^\top\bm x  - b_i)^2  \\
\mbox{ s.t. } & \displaystyle \sum_{i=1}^m s_i \geq m-r.
\end{array}
\end{align*}
%where the sum of the $m-r$ smallest squared residuals is minimum under the assumption that $r (\in [1, m))$ out of $m$ samples are outliers.
Note that the problem can be equivalently transformed into
  \begin{align}
\begin{array}{cl}
  \min\limits_{\bm{x}\in\mathbb R^{n}, \bm{z}\in\mathbb R^m} &
  \frac{1}{2}\|\bm A \bm x-\bm b - \bm z\|^2  \\
    \mbox{ s.t. } & \|\bm{z}\|_0\leq r,
\end{array}
\label{outlierdetect}
  \end{align}
  where $\bm A=\begin{bmatrix} \bm a_1  & \cdots &  \bm a_m \end{bmatrix}^\top$ and $\bm b=\begin{bmatrix}  b_1 & \cdots & b_m  \end{bmatrix}^\top$.
%  Any nonzero element (say, $j$) of the optimal solution $\bm z^*$ indicates that the corresponding sample is an outlier
%  and we must have $z_j^* = \bm a_j^\top \bm x^*- b_j$ to make
%  the residual zero.
  The above $\ell_0$-norm constrained problem was considered in \cite{TibshiraniTaylor2011}.
  More recently, an outlier detection problem using a nonconvex regularizer
  such as soft and hard thresholdings, SCAD \cite{Fan01}, etc., was proposed in \cite{SheOwen11}. In particular, they did not impose the constraint $\|\bm{z}\|_0\leq r$ directly.

  Most robust regression models make use of the squared loss function, i.e., $\Psi(\bm u) := \sum_{i=1}^m\psi_i(u_i)$ with $\psi_i(s) = \frac12(s-b_i)^2$. Alternatively, one can develop robust regression models based on the following loss functions, which are also commonly used in other branches of statistics: %$\Psi(\bm u)   := \sum_{i=1}^m\psi_i(u_i)$, where
\begin{itemize}
%  \item squared loss: $\psi_i(s) = \frac12(s - b_i)^2$,
  \item quadratic $\epsilon$-insensitive loss: $\psi_i(s) = \frac12(|s - b_i| - \epsilon)_+^2 $,
    where $(a)_+ := \max\{0,a\}$,
    for a given hyperparameter $\epsilon > 0$;
  \item  quantile squared loss: $\psi_i(s) =\frac\tau{2}  (s - b_i)_+^2 + \frac{(1-\tau)}2 (-s + b_i)_+^2$ for a given hyperparameter $0< \tau <1$;
\end{itemize}
and $s=\bm a_i^\top \bm x$ is assumed for regression problems.
These loss functions can be robustified by
incorporating a variable $z_i$ into them in the form of $\psi_i(\bm a_i^\top \bm x-z_i)$ as in the squared-loss model \eqref{outlierdetect}.
We can also add a regularization term, e.g., a nonconvex regularizer $P(\bm x)$ for inducing sparsity, to the robust regression problem \eqref{outlierdetect}.
This can improve the predictive error of the model by reducing the variability in the estimates of regression coefficients by shrinking the estimates towards zero. The resulting model takes the following from:
\begin{align}\label{trimmed_ls_vs}
\begin{array}{cl}
\displaystyle  \min_{\bm x\in\bR^n, \bm z\in\bR^m} & \displaystyle \sum_{i=1}^m\psi_i(\bm a_i^\top\bm x - z_i) + P(\bm x) \\
\mbox{ s.t. } & \displaystyle \|\bm z\|_0\le r.
\end{array}
\end{align}
In the case when $\psi_i$ is the squared loss function, the above model can be naturally referred to as the least trimmed squares regression with variable selection.

\subsection{A general model and algorithm}

In this section, we present a general model that covers the simultaneous sparse recovery and outlier detection models discussed in Section~\ref{sec:outlier_detect} for a large class of nonconvex regularizers, and discuss how the model can be solved by pDCA$_e$.

Specifically, we consider the following model:
\begin{equation}\label{outlier}
\min_{\bm x\in\bR^n, \bm z\in\bR^m}\ \Phi(\bm x, \bm z):=\Psi(\bm A\bm x - \bm z) + \delta_{\Omega}(\bm z) + {\cal J}_1(\bm x) - {\cal J}_2(\bm x),
\end{equation}
where $\Psi(\bm s):= \sum_{i=1}^m\psi_i(s_i)$ with $\psi_i:\bR\rightarrow [0,\infty)$ being convex with Lipschitz continuous gradient whose Lipschitz continuity modulus is $L_i$, $\bm A\in\bR^{m\times n}$, $\Omega = \{\bm z\in \bR^m: \|\bm z\|_0\le r\}$ for some positive integer $r$, ${\cal J}_1$ is a proper closed convex function and ${\cal J}_2$ is a \emph{continuous} convex function. In addition, we assume that $\Argmin \psi_i \neq\emptyset$ for each $i$ and that ${\cal J}_1 - {\cal J}_2$ is level-bounded. One can show that the squared loss function, the quadratic $\epsilon$-insensitive loss and the quantile squared loss function mentioned in Section~\ref{sec:5.1.2} satisfy the assumptions on $\psi_i$. Thus, when the regularizer $P$ in \eqref{robust_cs} or \eqref{trimmed_ls_vs} is level-bounded and can be written as the difference of a proper closed convex function and a continuous convex function, then the corresponding problem is a special case of \eqref{outlier}.

In order to apply the pDCA$_e$, we need to derive an explicit DC decomposition of the objective in \eqref{outlier} into the form of \eqref{problem}. To this end, we first note that $\nabla\Psi$ is Lipschitz continuous with a Lipschitz continuity modulus of $L_{\Psi}:= \max_{1\le i\le m}L_i$. Then we know that $g(\bm s) :=\frac {L_{\Psi}}2\|\bm s\|^2 - \Psi(\bm s)$ is convex and continuously differentiable. Hence,
\begin{equation}\label{out_DC}
\begin{split}
&\inf_{\bm z\in\bR^m} \Phi(\bm x, \bm z)  = \inf_{\bm z\in\bR^m}\Psi(\bm A\bm x - \bm z) + \delta_{\Omega}(\bm z) + {\cal J}_1(\bm x) - {\cal J}_2(\bm x)\\
& = \inf_{\bm z\in\Omega} \left[\frac{L_{\Psi}}2\|\bm A\bm x - \bm z\|^2 - g(\bm A\bm x - \bm z)\right] + {\cal J}_1(\bm x) - {\cal J}_2(\bm x)\\
& = \frac{L_{\Psi}}2\|\bm A\bm x\|^2 + {\cal J}_1(\bm x) - \underbrace{\sup_{\bm z\in\Omega}\left\{L_{\Psi}\langle\bm z, \bm A\bm x\rangle - \frac{L_{\Psi}}2\|\bm z\|^2 + g(\bm A\bm x - \bm z)\right\}}_{Q(\bm x)} - {\cal J}_2(\bm x).
\end{split}
\end{equation}
Now, notice that for each $\bm z\in \Omega$, the function $\bm x \mapsto L_{\Psi}\langle\bm z, \bm A\bm x\rangle - \frac{L_{\Psi}}2\|\bm z\|^2 + g(\bm A\bm x - \bm z)$ is convex and $Q$ is the pointwise supremum of these functions. Therefore, $Q$ is a convex function. In addition, one can see from \eqref{out_DC} that
\begin{equation}\label{Qrel}
  Q(\bm x) = \frac{L_\Psi}2\|{\bm A}{\bm x}\|^2 - \inf_{\bm z\in \Omega}\Psi({\bm A}\bm x - \bm z).
\end{equation}
In particular, $Q$ is a convex function that is finite everywhere, and is hence continuous. Using these observations, we can now rewrite
\eqref{outlier} as
\begin{equation}\label{dc_form}
    \min_{\bm x\in\bR^n} \frac{L_{\Psi}}2\|\bm A\bm x\|^2 + P_1(\bm x) - P_2(\bm x), \\
\end{equation}
where $P_1(\bm x) := {\cal J}_1(\bm x)$, and $P_2(\bm x) := Q(\bm x) + {\cal J}_2(\bm x)$ is a continuous convex function. This problem is in the form of \eqref{problem}, and its objective is level-bounded in view of \eqref{out_DC}, the level-boundedness of ${\cal J}_1-{\cal J}_2$ and the nonnegativity of $\Psi$. Hence, the pDCA$_e$ is applicable for solving it.

In each iteration of the pDCA$_e$, one has to compute the proximal mapping of $P_1 = {\cal J}_1$ and a subgradient of $P_2 = Q + {\cal J}_2$. Since $Q$ is continuous, it is well known that $\partial P_2(\bm x) = \partial Q(\bm x) + \partial {\cal J}_2(\bm x)$ for all $\bm x$. The ease of computation of the proximal mapping of ${\cal J}_1$ and a subgradient of ${\cal J}_2$ depends on the choice of regularizer, while a subgradient of $Q$ at $\bm x$ is readily computable using the observation that for any $\bar{\bm z} \in\Argmin_{\bm z\in\Omega}\Psi(\bm A\bm x - \bm z)$, we have
\begin{equation}\label{subdifferential_Q}
  L_{\Psi}\bm A^\top\bar{\bm z} + \bm A^\top\nabla g(\bm A\bm x - \bar{\bm z}) = L_{\Psi}\bm A^\top\bm A\bm x - \bm A^\top\nabla\Psi(\bm A\bm x - \bar{\bm z}) \in\partial Q(\bm x);
\end{equation}
this inclusion follows immediately from the definition of $Q$ in \eqref{out_DC} and the definition of convex subdifferential.

We are now ready to present the pDCA$_e$ for solving \eqref{dc_form} (and hence \eqref{outlier}) as Algorithm~\ref{alg2} below.
\begin{algorithm}
\caption{pDCA$_e$ for \eqref{outlier}:}
\label{alg2}
\begin{description}
\item [Input:] $\bm x^0\in\dom\,{\cal J}_1$, $\{\beta_k\}\subseteq [0,1)$ with $\sup_k\beta_k < 1$ and $L \ge L_{\Psi}\lambda_{\max}(\bm A^\top\bm A)$. Set $\bm x^{-1} = \bm x^0$.

\begin{description}
\item [for] $k = 0,1,2,\ldots$

\item [] take any $\bm\eta^{k+1}\in\partial {\cal J}_2(\bm x^k)$, $\bm z^{k+1}\in \Argmin\limits_{\bm z\in\Omega}\Psi(\bm A\bm x^k - \bm z)$ and set
\begin{equation*}%\label{ama_update}
  \left\{
   \begin{aligned}
  \bm u^k &= \bm x^k + \beta_k(\bm x^k - \bm x^{k-1}),  \\
  \bm v^k &= \bm A^\top\nabla\Psi(\bm A\bm x^k - \bm z^{k+1}) + L_{\Psi}\bm A^\top\bm A(\bm u ^k - \bm x^k)- \bm\eta^{k+1},\\
  \bm x^{k+1} &= \argmin\limits_{\bm x\in \bR^n}\left\{\langle\bm v^k, \bm x\rangle + \frac L2\|\bm x - \bm u^k\|^2 + {\cal J}_1(\bm x)\right\}.  \\
   \end{aligned}
  \right.
\end{equation*}
\item [end for]
\end{description}
\end{description}
\end{algorithm}
Notice that this algorithm is just Algorithm~\ref{alg1} applied to \eqref{dc_form} with a subgradient of $Q$ computed as in \eqref{subdifferential_Q} in each step.
The following lemma gives a closed-form solution for the $\bm z$-update in Algorithm~\ref{alg2}, and is an immediate corollary of \cite[Proposition 3.1]{Lu13}.
\begin{lemma}\label{z_sub_lemma}
Fix any $\tilde{\bm z}\in\Argmin_{\bm z\in \bR^m}\Psi(\bm z) = \Argmin_{\bm z\in \bR^m}\sum_{i=1}^m\psi_i(z_i)$ and let $\tilde{\bm z}^k = \bm A\bm x^k - \tilde{\bm z}$. Let $I^*\subseteq\{1,\ldots,m\}$ be an index set corresponding to any $r$ largest values of $\{\psi_i([\bm A\bm x^k]_i) - \psi_i(\tilde{z}_i)\}_{i=1}^m$ and set
\begin{equation*}
  z^{k+1}_i = \begin{cases}
  \tilde{z}_i^k & {\rm if}\ i \in I^*,\\
  0 & {\rm otherwise}.
\end{cases}
\end{equation*}
Then we have
\begin{equation*}
  \bm z^{k+1}\in \Argmin\limits_{\bm z\in\Omega}\Psi(\bm A\bm x^k - \bm z).
\end{equation*}
\end{lemma}

In the last theorem of this section, we show that for problem \eqref{dc_form} (equivalently, \eqref{outlier}) with many commonly used loss functions $\psi_i$ and regularizers ${\cal J}_1 - {\cal J}_2$, the corresponding potential function $E$ is a KL function with exponent $\frac12$. This together with the discussion at the end of Section~\ref{sec3} reveals that the pDCA$_e$ is locally linearly convergent when applied to these models.
In the proof below, for notational simplicity, for a positive integer $m$, we let $S_m$ denote the set of all possible permutations of $\{1,\ldots,m\}$.

\begin{theorem}\label{thm_Q}
Let $Q$ be given in \eqref{Qrel}, with $\psi_i$ taking one of the following forms:
\begin{enumerate}[{\rm (i)}]
  \item squared loss: $\psi_i(s) = \frac12(s - b_i)^2$, $b_i\in\bR$;
  \item squared hinge loss: $\psi_i(s) = \frac12(1 - b_is)_+^2$, $b_i\in\{1, -1\}$;
  \item quadratic $\epsilon$-insensitive loss: $\psi_i(s) = \frac12(|s - b_i| - \epsilon)_+^2 $, $\epsilon > 0$ and $b_i\in\bR$;
  \item quantile squared loss: $\psi_i(s) = \frac{\tau}2  (s - b_i)_+^2 + \frac{(1-\tau)}2 (-s + b_i)_+^2$, $0 <\tau < 1$ and $b_i\in\bR$.
\end{enumerate}
Then $Q$ is a convex piecewise linear-quadratic function.
Suppose in addition that ${\cal J}_1$ and ${\cal J}_2$ are convex piecewise linear-quadratic functions. Then the function $E$ in \eqref{merit_fun} corresponding to \eqref{dc_form} is a KL function with exponent $\frac12$.
\end{theorem}

\begin{proof}
We first prove that $Q$ is convex piecewise linear-quadratic when $\psi_i$ is chosen as one of the four loss functions.
We start with (i). Clearly, $L_\Psi = 1$ and we have from \eqref{Qrel} that
\begin{equation}\label{Q_form}
   Q(\bm x) = \frac{1}2\|\bm A\bm x\|^2 - \inf_{\bm z\in\Omega}\Psi(\bm A\bm x - \bm z) = \frac12\|\bm A\bm x\|^2 - \inf_{\bm z\in\Omega}\left\{\sum_{i=1}^m\frac12([\bm A\bm x - \bm b]_i -z_i)^2\right\}.
\end{equation}
Let $\mathcal{I}$ be an index set corresponding to any $r$ largest entries of $\bm A\bm x - \bm b$ in magnitude. We then see from Lemma~\ref{z_sub_lemma} that if
\begin{equation*}
  z_i^* = \begin{cases}
[\bm A\bm x - \bm b]_i & {\rm if} \ i\in\mathcal{I},\\
0 & {\rm otherwise,}
\end{cases}
\end{equation*}
then $\bm z^*$ attains the infimum in \eqref{Q_form}. Thus, we have
\begin{equation*}
  Q(\bm x) = \frac12\|\bm A\bm x\|^2 - \sum_{i=r+1}^m\frac12\left([\bm A\bm x - \bm b]_{[i]}\right)^2,
\end{equation*}
where $[\bm A\bm x - \bm b]_{[i]}$ denotes the $i$th largest entry of $\bm A\bm x - \bm b$ in magnitude. Notice that for each fixed permutation $\sigma\in S_m$, the set
\[
\Omega_{\sigma} = \left\{\bm x: \left|[\bm A\bm x - \bm b]_{\sigma(1)}\right|\ge \left|[\bm A\bm x - \bm b]_{\sigma(2)}\right|\ge\ldots\ge\left|[\bm A\bm x - \bm b]_{\sigma(m)}\right|\right\}
\]
is a union of finitely many polyhedra and the restriction of $Q$ onto $\Omega_{\sigma}$ is a quadratic function. Moreover, $\bigcup_{\sigma\in S_m}\Omega_{\sigma} = \bR^n$. Thus, $Q$ is a piecewise linear-quadratic function when $\psi_i$ takes the form in (i).

Then we consider case (ii). Again, $L_\Psi = 1$, and we have from \eqref{Qrel}, $b_i\in\{1, -1\}$ and Lemma~\ref{z_sub_lemma} that
\begin{equation*}
\begin{split}
  Q(\bm x) &= \frac12\|\bm A\bm x\|^2 - \inf_{\bm z\in\Omega}\Psi(\bm A\bm x - \bm z) = \frac12\|\bm A\bm x\|^2 - \inf_{\bm z\in\Omega}\left\{\sum_{i=1}^m\frac12(1 - b_i[\bm A\bm x - \bm z]_i)_+^2\right\}\\
           &= \frac12\|\bm A\bm x\|^2 - \inf_{\bm z\in\Omega}\left\{\sum_{i=1}^m\frac12(1 - b_i[\bm A\bm x]_i - z_i)_+^2\right\} = \frac12\|\bm A\bm x\|^2 - \sum_{i=r+1}^m\frac12\left(\left[\bm e- \bm b\circ\bm A\bm x\right]_{(i)}\right)_+^2
\end{split}
\end{equation*}
where $\left[\bm e- \bm b\circ\bm A\bm x\right]_{(i)}$ the $i$th largest entry of $\bm e- \bm b\circ\bm A\bm x$, where $\bm e\in \bR^m$ is the vector of all ones. Note that for each fixed permutation $\sigma\in S_m$, the set
\begin{equation*}
  \Omega_{\sigma} = \left\{\bm x: [\bm e- \bm b\circ\bm A\bm x]_{\sigma(1)} \ge [\bm e- \bm b\circ\bm A\bm x]_{\sigma(2)}\ge\ldots\ge [\bm e- \bm b\circ\bm A\bm x]_{\sigma(m)}\right\}
\end{equation*}
is a polyhedron and the restriction of $Q$ onto $\Omega_{\sigma}$ is a piecewise linear-quadratic function. Moreover, $\bigcup_{\sigma\in S_m}\Omega_{\sigma} = \bR^n$. Thus, $Q$ is a piecewise linear-quadratic function when $\psi_i$ takes the form in (ii).

Next, we turn to case (iii). Again, $L_\Psi = 1$, and we have from \eqref{Qrel} and Lemma~\ref{z_sub_lemma} that
\begin{equation*}
  Q(\bm x) = \frac12\|\bm A\bm x\|^2 - \inf_{\bm z\in\Omega}\left\{\frac12\sum_{i=1}^m(\left|[\bm A\bm x - \bm z]_i - b_i\right| - \epsilon)_+^2\right\} = \frac12\|\bm A\bm x\|^2 - \sum_{i=r+1}^m\frac12\left(\left|[\bm A\bm x - \bm b]_{[i]}\right| - \epsilon\right)_+^2,
\end{equation*}
where $[\bm A\bm x - \bm b]_{[i]}$ denotes the $i$th largest entry of $\bm A\bm x - \bm b$ in magnitude.
Using a similar argument as above, one can see that $Q$ is a piecewise linear-quadratic function.

Finally, we consider case (iv). Define $w_i = \sqrt{\frac{\tau}2} ([\bm A\bm x]_i - b_i)_+ + \sqrt{\frac{1-\tau}2} (- [\bm A\bm x]_i + b_i)_+$, $i=1,\ldots,m$. Then
\[
w_i^2 = \frac{\tau}2 ([\bm A\bm x]_i - b_i)^2_+ + \frac{1-\tau}2 (- [\bm A\bm x]_i + b_i)^2_+,
\]
and for all $i$, $j$, it holds that $w_i\ge w_j$ if and only if $w_i^2\ge w_j^2$.
Since we can take $L_\Psi = 1$, we have from these and \eqref{Qrel} that
\begin{align*}
  Q(\bm x) &= \frac12\|\bm A\bm x\|^2 - \inf_{\bm z\in\Omega}\left\{\sum_{i=1}^m\frac{\tau}2 ([\bm A\bm x]_i - z_i - b_i)_+^2 + \frac{1-\tau}2 (- [\bm A\bm x]_i + z_i + b_i)_+^2\right\}\\
           &= \frac12\|\bm A\bm x\|^2 - \sum_{i=r+1}^m (w_{\{i\}})^2,
\end{align*}
where $w_{\{i\}}$ denotes the $i$th largest element of $\{w_i\}_{i=1,\ldots,m}$. For each fixed permutation $\sigma\in S_m$, we define a set
\begin{equation*}
  \Omega_{\sigma} : = \left\{\bm x: w_{\sigma(1)}\ge w_{\sigma(2)}\ge\ldots\ge w_{\sigma(m)}\right\}.
\end{equation*}
Notice that the restriction of $Q$ onto $\Omega_{\sigma}$ is a piecewise linear-quadratic function. Moreover, $\Omega_{\sigma}$ can be written as a union of finitely many polyhedra and $\bigcup_{\sigma\in S_m}\Omega_{\sigma} = \bR^n$. Thus, $Q$ is a piecewise linear-quadratic function.

Finally, we show that $E$ is a KL function with exponent $\frac12$ under the additional assumption that ${\cal J}_1$ and ${\cal J}_2$ are convex piecewise linear-quadratic functions. Notice from \eqref{dc_form} and \eqref{merit_fun} that
\begin{eqnarray}\label{Ecite1}
  E(\bm x, \bm y, \bm w) = \frac12\|\bm A\bm x\|^2 + {\cal J}_1(\bm x) - \langle\bm x, \bm y\rangle + (Q + {\cal J}_2)^*(\bm y) + \frac{L}2\|\bm x - \bm w\|^2.
\end{eqnarray}
Since $Q$ and ${\cal J}_2$ are convex piecewise linear-quadratic functions, we know from \cite[Exercise 10.22]{RW98} and \cite[Theorem 11.14]{RW98} that $(Q + {\cal J}_2)^*$ is also a piecewise linear-quadratic function. Hence, ${\cal J}_1(\bm x) + (Q + {\cal J}_2)^*(\bm y)$ is also a piecewise linear-quadratic functions and can be written as
\begin{equation*}
  {\cal J}_1(\bm x) + (Q + {\cal J}_2)^*(\bm y) = \min_{1\le i\le M}\left\{g_i(\bm x, \bm y) + \delta_{C_i}(\bm x, \bm y)\right\},
\end{equation*}
where $M > 0$ is an integer, $g_i$ are quadratic functions and $C_i$ are polyhedra. Then we have
\begin{equation*}
  E(\bm x, \bm y, \bm w) = \min_{1\le i\le M}\left\{\frac12\|\bm A\bm x\|^2 - \langle\bm x, \bm y\rangle + \frac{L}2\|\bm x - \bm w\|^2 + g_i(\bm x, \bm y) + \delta_{C_i}(\bm x, \bm y)\right\}.
\end{equation*}
Moreover, this function is continuous in its domain because, according to \eqref{Ecite1}, it is the sum of piecewise linear-quadratic functions which are continuous in their domains \cite[Proposition~10.21]{RW98}. Thus, by \cite[Corollary 5.2]{Li_TK_Calculus}, this function is a KL function with exponent $\frac12$. This completes the proof.
\end{proof}

\section{Numerical simulations}\label{sec6}
%\subsection{Regression problems with outliers}
In this section, we perform numerical experiments to explore the performance of pDCA$_e$ in some specific simultaneous sparse recovery and outlier detection problems. All experiments are performed in Matlab R2015b on a 64-bit PC with an Intel(R) Core(TM) i7-4790 CPU (3.60GHz) and 32GB of RAM.

We consider the following special case of \eqref{outlier} with the least trimmed squares loss function and the Truncated $\ell_1$ regularizer:
\begin{equation}\label{ls_trun}
\min_{\bm x\in\bR^n, \bm z\in\bR^m}\ \Phi_{\rm trc}(\bm x, \bm z):= \frac12\|\bm A\bm x - \bm z - \bm b\|^2 + \delta_{\Omega}(\bm z) + \lambda\|\bm x\|_1 - \lambda\mu\sum_{i=1}^{p}|x_{[i]}|,
\end{equation}
where $\bm A\in\bR^{m\times n}$, $\bm b\in\bR^m$, $\Omega = \{\bm z\in \bR^m: \|\bm z\|_0\le r\}$, $\mu\in(0,1)$, $\lambda > 0$ is the regularization parameter, $p < n$ is a positive integer and $x_{[i]}$ denotes the $i$th largest entry of $\bm x$ in magnitude. One can see that $\Phi_{\rm trc}$ takes the form of \eqref{outlier}, where $\psi_i(s) = \frac12(s - b_i)^2$, ${\cal J}_1(\bm x) = \lambda\|\bm x\|_1$ and ${\cal J}_2(\bm x) = \lambda\mu\sum_{i=1}^{p}|x_{[i]}|$ with ${\cal J}_1 - {\cal J}_2$ being level-bounded. In this case, $L_\Psi = 1$, and we can rewrite \eqref{ls_trun} in the form \eqref{problem} (see also \eqref{dc_form}):
\begin{equation}\label{dc_num}
    \min_{\bm x\in\bR^n}\ F_{\rm trc}(\bm x) := \underbrace{\frac12\|\bm A\bm x\|^2}_{f(\bm x)} + \underbrace{\lambda\|\bm x\|_1}_{P_1(\bm x)} - \underbrace{\left(\lambda\mu\sum_{i=1}^{p}|x_{[i]}| + Q(\bm x)\right)}_{P_2(\bm x)},
\end{equation}
where $Q$ is defined in \eqref{Qrel}. %Then $F_{\rm trc}$ is level-bounded.

Notice that ${\cal J}_1$ and ${\cal J}_2$ are piecewise linear-quadratic functions. By Theorem~\ref{thm_Q}, we see that for the function $F_{\rm trc}$ given in \eqref{dc_num}, the corresponding function $E$ in \eqref{merit_fun} is a KL function with exponent $\frac12$. Thus, we conclude from Theorem~\ref{thm2} and the discussion at the end of Section~\ref{sec3} that the sequence $\{\bm x^k\}$ generated by pDCA$_e$ for solving \eqref{dc_num} converges locally linearly to a stationary point of $F_{\rm trc}$ in \eqref{dc_num}.

In our experiments below, we compare pDCA$_e$ with NPG$_{\rm major}$ \cite{Liu17} for solving \eqref{dc_num}. We discuss the implementation details below. For the ease of exposition, we introduce an auxiliary function:
\begin{equation*}
  \Xi (\bm x, \bm y) = f(\bm x) + P_1(\bm x) - \langle\bm x, \bm y\rangle + P_2^*(\bm y),
\end{equation*}
where $f$, $P_1$ and $P_2$ are defined in \eqref{dc_num}. Notice that $\bm 0\in \partial \Xi(\bar{\bm x},\bar{\bm y})$ for some $\bar{\bm y}$ if and only if $\bar{\bm x}$ is a stationary point of $F_{\rm trc}$ in \eqref{dc_num}, i.e., $\bm 0 \in \nabla f(\bar{\bm x}) + \partial P_1(\bar{\bm x}) - \partial P_2(\bar{\bm x})$. We will employ $\Xi$ in the design of termination criterion for the algorithms.

\paragraph{pDCA$_e$.} We apply Algorithm~\ref{alg1} to \eqref{dc_num}, with a subgradient of $Q$ computed as in \eqref{subdifferential_Q} in each step.\footnote{As mentioned before, with this choice of subgradient in Algorithm~\ref{alg1}, the algorithm is equivalent to Algorithm~\ref{alg2}.} In our experiments below, as in \cite[Section~5]{Wen16}, we set $\bm x^0 = \bm 0$,  $L = \lambda_{\max}(\bm A^\top\bm A)$ and start with $\theta_{-1} = \theta_0 = 1$, recursively define for $k\ge 0$ that
\begin{equation*}
  \beta_k = \theta_k(\theta_{k-1}^{-1} - 1)\ \ \mbox{with}\ \ \theta_{k+1} = \frac2{1 + \sqrt{1 + 4/\theta_k^2}}.
\end{equation*}
We then reset $\theta_{-1} = \theta_0 = 1$ every 200 iterations. To derive a reasonable termination criterion, we first note from the first-order optimality condition of the $\bm x$-update in \eqref{iter_update} that
\begin{equation*}
  -\nabla f(\bm u^{k-1}) + \bm\xi^k - L(\bm x^k - \bm u^{k-1})\in\partial P_1(\bm x^k).
\end{equation*}
This together with $\bm x^{k-1}\in\partial P_2^*(\bm\xi^k)$ and
\begin{equation*}
  \partial \Xi(\bm x^k, \bm\xi^k) = \begin{bmatrix}\nabla f(\bm x^k) + \partial P_1(\bm x^k) - \bm\xi^k\\-\bm x^k + \partial P_2^*(\bm\xi^k)\end{bmatrix}
\end{equation*}
implies that
\begin{equation*}
  \begin{bmatrix}\nabla f(\bm x^k) - \nabla f(\bm u^{k-1}) - L(\bm x^k - \bm u^{k-1})\\-\bm x^k + \bm x^{k-1}\end{bmatrix}\in\partial \Xi(\bm x^k, \bm\xi^k).
\end{equation*}
Thus, we terminate the algorithm when
\begin{equation*}
\sqrt{\left(\sqrt{L}\|\bm A(\bm x^k - \bm u^{k-1})\| + L\|\bm x^k - \bm u^{k-1}\|\right)^2 + \|\bm x^k - \bm x^{k-1}\|^2} < 10^{-4}\max\{1, \|\bm x^k\|\},
\end{equation*}
so that we have $\dist(0, \partial \Xi(\bm x^k, \bm\xi^k))< 10^{-4}\max\{1, \|\bm x^k\|\}$.

\paragraph{NPG$_{\rm major}$.} We solve \eqref{dc_num} by the NPG$_{\rm major}$ algorithm described in \cite[Apendix A, Algorithm 2]{Liu17}, which is basically the proximal DCA incorporated with a nonmonotone linesearch scheme. Following the notation there, we apply the method with $h(\bm x) = f(\bm x)$, $P(\bm x) = P_1(\bm x)$ and $g(\bm x) = P_2(\bm x)$, and set ${\bm x}^0 = {\bm 0}$, $\tau = 2$, $c = 10^{-4}$, $M = 4$, $L_0^0 = 1$, $L_{\min} = 10^{-8}$, $L_{\max} = 10^8$ and
\[
L_k^0 = \begin{cases}
\min\left\{ \max\left\{\frac{{\bm s^k}^\top\bm y^k}{\|\bm s^k\|^2}, L_{\min}\right\}, L_{\max}\right\} & {\rm if}\ {\bm s^k}^\top\bm y^k \ge 10^{-12},\\
\min\left\{\max\left\{\frac{\bar L_{k-1}}2, L_{\min}\right\}, L_{\max}\right\} & {\rm otherwise,}
\end{cases}
\]
for $k\ge 1$; here, $\bar{L}_{k-1}$ is determined in \cite[Apendix A, Algorithm 2, Step 2]{Liu17}, $\bm s^k = \bm x^k - \bm x^{k-1}$ and $\bm y^k = \bm A^\top\left(\bm A\bm x^k - \bm z^{k+1}\right) - \bm A^\top\left(\bm A\bm x^{k-1} - \bm z^k\right)$, \footnote{Note that $\bm A^\top(\bm A\bm x^k - \bm z^{k+1}) = \nabla h({\bm x}^k) - {\bm \zeta}^k + {\bm \eta}^{k+1}$ by our choice of ${\bm \zeta}^k$ in the subproblem \eqref{subp_npg}. Thus, this quantity can be obtained as a by-product when solving \eqref{subp_npg}.} where $\bm z^{k+1}$ is chosen from $\Argmin_{\bm z\in\Omega}\Psi(\bm A\bm x^k - \bm z)$. We choose ${\bm \eta}^{k+1}\in \partial {\cal J}_2({\bm x}^k)$, set $\bm\zeta^k := {\bm A}^\top{\bm z}^{k+1} + {\bm \eta}^{k+1}\in\partial g(\bm x^k)$ \footnote{Notice from $\psi_i(s) = \frac12(s-b_i)^2$, \eqref{subdifferential_Q} and the definition of $\bm z^{k+1}$ that $\bm A^\top \bm z^{k+1}\in \partial Q(\bm x^k)$. This together with ${\bm \eta}^{k+1}\in \partial {\cal J}_2({\bm x}^k)$ and $g = {\cal J}_2 + Q$ gives $\bm \zeta^k\in \partial g(\bm x^k)$.} and solve subproblems in the following form in each iteration; see \cite[Eq.~46]{Liu17}:
\begin{equation}\label{subp_npg}
 \min_{\bm x\in \bR^n}\left\{\langle\nabla h(\bm x^k) - \bm\zeta^k,  \bm x - \bm x^k\rangle + \frac{L_k}2\| \bm x - \bm x^k\|^2 + P(\bm x)\right\}.
\end{equation}
The above subproblem has a closed-form solution, thanks to $P(\bm x) = P_1(\bm x) = \lambda\|\bm x\|_1$. Finally, to derive a reasonable termination criterion, we note from the first-order optimality condition of \eqref{subp_npg} (with $L_k = \bar{L}_k$ determined in \cite[Apendix A, Algorithm 2, Step 2]{Liu17}) that
\begin{equation*}
  -\nabla h(\bm x^{k-1}) + \bm\zeta^{k-1} - \bar L_{k-1}(\bm x^k - \bm x^{k-1})\in\partial P(\bm x^k).
\end{equation*}
On the other hand, we have (recalling that $f = h$, $P_1 = P$ and $P_2 = g$) that
\begin{equation*}
  \partial \Xi(\bm x^k, \bm\zeta^{k-1}) = \begin{bmatrix}\nabla h(\bm x^k) + \partial P(\bm x^k) - \bm\zeta^{k-1}\\-\bm x^k + \partial g^*(\bm\zeta^{k-1})\end{bmatrix}.
\end{equation*}
These together with $\bm x^{k-1}\in\partial g^*(\bm \zeta^{k-1})$ give
\begin{equation*}
\begin{bmatrix}\nabla h(\bm x^k) - \nabla h(\bm x^{k-1}) - \bar{L}_{k-1}(\bm x^k - \bm x^{k-1})\\ - \bm x^k + \bm x^{k-1}\end{bmatrix} \in \partial \Xi(\bm x^k, \bm\zeta^{k-1}),
\end{equation*}
Thus, we terminate the algorithm when
\begin{equation*}
\sqrt{\left(\sqrt{L}\|\bm A(\bm x^k - \bm x^{k-1})\| + \bar{L}_{k-1}\|\bm x^k - \bm x^{k-1}\|\right)^2  + \|\bm x^k - \bm x^{k-1}\|^2} < 10^{-4}\max\{1, \|\bm x^k\|\},
\end{equation*}
so that we have $\dist(0, \partial \Xi(\bm x^k, \bm\zeta^{k-1}))< 10^{-4}\max\{1, \|\bm x^k\|\}$.

\paragraph{Simulation results:}  We first generate a matrix $\bm A\in\bR^{(m+t)\times n}$ with i.i.d. standard Gaussian entries and then normalize each column of $\bm A$ to have unit norm. Next, we let $\bm x_{\rm true}\in\bR^n$ be an $s$-sparse vector with $s$ i.i.d. standard Gaussian entries at random positions. Moreover, we choose $\bm z\in\bR^{m+t}$ to be the vector with the last $t$ entries being 8 and others being 0. The vector $\bm b$ is then generated as $\bm b = \bm A\bm x_{\rm true} - \bm z + \sigma\bm\epsilon$, where $\sigma > 0$ is a noise factor and $\bm\epsilon\in\bR^{m+t}$ is a random vector with i.i.d. standard Gaussian entries.

In our numerical test, we consider three different values for $\lambda$: $5\times 10^{-3}$, $10^{-3}$ and $5\times 10^{-4}$ in \eqref{ls_trun}. For the same $\lambda$ value, for each $(m, n, s, t) = (600i, 3000i, 150i, 30i)$, $i = 1, 2, 3$, we generate 20 random instances as described above with $\sigma = 10^{-2}$ and solve the corresponding \eqref{ls_trun} with $\mu=0.99$, $p = 0.8s$ and $r \in \{ t, 1.1t\}$. Our computational results are reported in Tables~\ref{table_1} and \ref{table_2}. We present the number of iterations (iter), the best function values attained till termination (fval) and CPU times in seconds (CPU), averaged over the 20 random instances. One can see that pDCA$_e$ is always faster than NPG$_{\rm major}$ and returns slightly smaller function values.

Finally, to illustrate the ability of recovering the original sparse solution by solving \eqref{ls_trun} with the chosen parameters $p$, $r$ and $\lambda$, we also present in the tables the root-mean-square-deviation (\rmsd) $\frac1{\sqrt{n}}\|\bm x_{{\rm pDCA_e}} - \bm x_{\rm true}\|$ for the approximate solution $\bm x_{{\rm pDCA_e}}$ returned by pDCA$_e$ that corresponds to the best attained function value, averaged over the 20 random instances. The relatively small {\rmsd}'s obtained suggest that our method is able to recover the original sparse solution approximately. As a further illustration, we also plot $\bm x_{{\rm pDCA_e}}$ (marked by asterisks) against $\bm x_{\rm true}$ (marked by circles) in Figure~\ref{compare} below for a randomly generated instance with $m = 1800$, $n = 9000$, $s = 450$ and $t = 90$ (i.e., $i = 3$). We use $\mu = 0.99$, $\lambda = 5\times 10^{-4}$, $p = 0.8s$, and set $r = t$ and $r = 1.1t$ in Figures~\ref{compare}(a) and \ref{compare}(b), respectively. One can see that the recovery results are similar even though the $r$ used are different.
\begin{figure}[h]
\caption{Recovery comparison for different $r$.}
\label{compare}
\begin{subfigure}{.5\textwidth}
  \centering
  \includegraphics[width=1.1\linewidth]{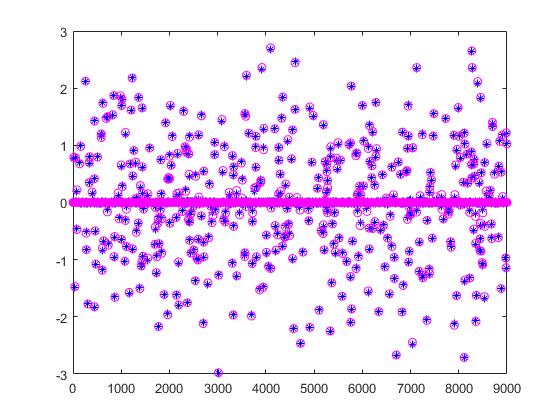}
  \caption{Recovery result for $r = t$.}
\end{subfigure}%
\begin{subfigure}{.5\textwidth}
  \centering
  \includegraphics[width=1.1\linewidth]{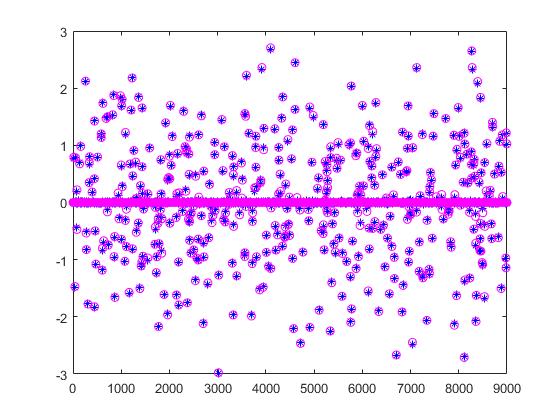}
  \caption{Recovery result for $r = 1.1t$.}
\end{subfigure}
\end{figure}

\setlength{\tabcolsep}{2.8pt}
\begin{table}[H]
\caption{Numerical results for regression problem with $p = 0.8s$ and $r = t$.}
\label{table_1}
\centering
\begin{tabular}{c c c c c c c c c c c c}
\hline
\multirow{2}{*}{$\lambda$} & \multirow{2}{*}{$m$} & \multirow{2}{*}{$n$} & \multirow{2}{*}{$s$}  & \multirow{2}{*}{$t$}  & \multirow{2}{*}{\rmsd} & \multicolumn{2}{c}{$\mathrm{iter}$} & \multicolumn{2}{c}{$\mathrm{fval}$} & \multicolumn{2}{c}{CPU} \\
%\cline{6-11}
      &      &      &     &    &      & pDCA$_e$ & NPG &  pDCA$_e$ & NPG & pDCA$_e$ & NPG\\
\hline
      &   600 &  3000 &   150 &    30 & 5.0e-03 &   431 &   468 & 3.6365e-02 & 3.6380e-02 &   0.7 &   1.1   \\
5e-03 &  1200 &  6000 &   300 &    60 & 4.9e-03 &   422 &   460 & 7.1070e-02 & 7.1093e-02 &   2.9 &   4.1   \\
      &  1800 &  9000 &   450 &    90 & 5.0e-03 &   418 &   439 & 1.0485e-01 & 1.0491e-01 &   6.2 &   8.5   \\ \hline
      &   600 &  3000 &   150 &    30 & 5.4e-03 &  1276 &  1966 & 7.4019e-03 & 7.4293e-03 &   2.1 &   4.6   \\
1e-03 &  1200 &  6000 &   300 &    60 & 5.5e-03 &  1254 &  1956 & 1.5032e-02 & 1.5092e-02 &   8.5 &  18.1   \\
      &  1800 &  9000 &   450 &    90 & 5.5e-03 &  1298 &  2013 & 2.2594e-02 & 2.2667e-02 &  18.9 &  39.7   \\ \hline
      &   600 &  3000 &   150 &    30 & 6.0e-03 &  2361 &  3844 & 3.9910e-03 & 4.0136e-03 &   3.9 &   9.0   \\
5e-04 &  1200 &  6000 &   300 &    60 & 5.8e-03 &  2367 &  3890 & 7.5756e-03 & 7.6306e-03 &  16.0 &  36.1   \\
      &  1800 &  9000 &   450 &    90 & 5.8e-03 &  2311 &  3841 & 1.1407e-02 & 1.1523e-02 &  33.7 &  76.5   \\ \hline
\end{tabular}
\end{table}

\setlength{\tabcolsep}{2.8pt}
\begin{table}[h]
\caption{Numerical results for regression problem with $p = 0.8s$ and $r = 1.1t$.}
\label{table_2}
\centering
\begin{tabular}{c c c c c c c c c c c c}
\hline
\multirow{2}{*}{$\lambda$} & \multirow{2}{*}{$m$} & \multirow{2}{*}{$n$} & \multirow{2}{*}{$s$}  & \multirow{2}{*}{$t$}  & \multirow{2}{*}{\rmsd} & \multicolumn{2}{c}{$\mathrm{iter}$} & \multicolumn{2}{c}{$\mathrm{fval}$} & \multicolumn{2}{c}{CPU} \\
%\cline{6-11}
      &      &      &     &    &         & pDCA$_e$ & NPG &  pDCA$_e$ & NPG & pDCA$_e$ & NPG\\
\hline
      &   600 &  3000 &   150 &    30 & 5.1e-03 &   461 &   487 & 3.5891e-02 & 3.6026e-02 &   0.8 &   1.1   \\
5e-03 &  1200 &  6000 &   300 &    60 & 5.0e-03 &   454 &   483 & 7.0182e-02 & 7.0291e-02 &   3.1 &   4.4   \\
      &  1800 &  9000 &   450 &    90 & 5.1e-03 &   447 &   469 & 1.0346e-01 & 1.0374e-01 &   6.6 &   9.1   \\ \hline
      &   600 &  3000 &   150 &    30 & 5.4e-03 &  1530 &  2067 & 7.3386e-03 & 7.3593e-03 &   2.6 &   4.9   \\
1e-03 &  1200 &  6000 &   300 &    60 & 5.6e-03 &  1485 &  2053 & 1.4881e-02 & 1.4954e-02 &  10.1 &  19.1   \\
      &  1800 &  9000 &   450 &    90 & 5.5e-03 &  1550 &  2114 & 2.2397e-02 & 2.2483e-02 &  22.6 &  41.8   \\ \hline
      &   600 &  3000 &   150 &    30 & 6.0e-03 &  2837 &  4114 & 3.9613e-03 & 3.9972e-03 &   4.7 &   9.7   \\
5e-04 &  1200 &  6000 &   300 &    60 & 5.8e-03 &  2874 &  4061 & 7.5225e-03 & 7.5799e-03 &  19.4 &  37.7   \\
      &  1800 &  9000 &   450 &    90 & 5.9e-03 &  2778 &  4027 & 1.1384e-02 & 1.1449e-02 &  40.5 &  80.1   \\ \hline
\end{tabular}
\end{table}

\end{document}